\documentclass{amsart}
\usepackage{amssymb,url,xspace, amsthm}
\usepackage{bigints}
\usepackage{multirow}
\usepackage{hyperref}
\usepackage[all]{xy}

\newtheorem{theorem}{Theorem}[section]
\newtheorem{lemma}[theorem]{Lemma}
\newtheorem{props}[theorem]{Proposition}
\newtheorem{corol}[theorem]{Corollary}

\theoremstyle{definition}
\newtheorem{definition}[theorem]{Definition}
\newtheorem{example}[theorem]{Example}

\newtheorem{question}[theorem]{Question}
\newtheorem*{questionPalais*}{Palais' problem for operator spaces}

\theoremstyle{remark}

\numberwithin{equation}{section}

\newcommand{\vertiii}[1]{{\left\vert\kern-0.25ex\left\vert\kern-0.25ex\left\vert #1 
    \right\vert\kern-0.25ex\right\vert\kern-0.25ex\right\vert}}

\def\block(#1,#2)#3{\multicolumn{#2}{c}{\multirow{#1}{*}{$ #3 $}}}

\DeclareMathOperator{\Ima}{Im}

\begin{document}

\title{Twisting Operator Spaces}

\author{Willian Hans Goes Corr\^ea}
\address{Departamento de Matem\'atica, Instituto de Matem\'atica e Estat\'istica, Universidade de S\~ao Paulo, Rua do Mat\~ao 1010, 05508-090 S\~ao Paulo SP, Brazil}
\email{willhans@ime.usp.br}
\thanks{The present work was partially supported by CAPES, Coordena\c{c}\~{a}o de Aperfei\c{c}oamento de Pessoal de N\'{i}vel Superior, grant 1328372, and by CNPq, Conselho Nacional de Desenvolvimento Cient\'{i}fico  e Tecnol\'{o}gico - Brazil, grant 140413/2016-2}



\date{}

\dedicatory{}

\begin{abstract}
In this work we study the following three space problem for operator spaces: if $X$ is an operator space with base space isomorphic to a Hilbert space and $X$ contains a completely isomorphic copy of the operator Hilbert space $OH$ with respective quotient also completely isomorphic to $OH$, must $X$ be completely isomorphic to $OH$? This problem leads us to the study of short exact sequences of operator spaces, more specifically those induced by complex interpolation, and their splitting. We show that the answer to the three space problem is negative, giving two different solutions.
\end{abstract}

\maketitle

\tableofcontents

\section{Introduction}
\subsection{Overview}
In this work we are interested in solving a version of Palais' problem for operator spaces. Palais' problem asks if being isomorphic to a Hilbert space is a 3-space property, that is, if $X$ is a Banach space with a subspace $Y$ such that $Y$ and $X/Y$ are isomorphically Hilbert spaces, must $X$ be itself isomorphic to a Hilbert space? 

Palais' problem was answered in the negative by Enflo, Lindenstrauss and Pisier \cite{Enflo01}, who gave the first example of a Banach space that contains a copy of $\ell_2$ with $\ell_2$ as quotient, and which is not isomorphic to a Hilbert space. Some years later, Kalton and Peck \cite{Kalton01} solved this problem in a different way, and showed the equivalence between twisted sums of Banach spaces $Y$ and $Z$ and a class of nonlinear maps from $Z$ to $Y$, called quasilinear.

Recall that a twisted sum of Banach spaces is a short exact sequence
  \[
  \xymatrix{ 0 \ar[r] & Y \ar[r] & X
    \ar[r] & Z\ar[r]  &0}
  \]
  
\noindent where $Y$ and $Z$ are Banach spaces, $X$ is a quasi-Banach space, and the arrows are bounded linear maps. We refer either to the exact sequence or to $X$ as a twisted sum of $Y$ and $Z$ (in that order). If $Y = Z$, $X$ is also called an extension of $Y$. General conditions guarantee that $X$ is actually isomorphic to a Banach space, for example, if $Y$ and $Z$ are $B-$convex.

Kalton-Peck's space $Z_2$ appears as a twisted sum
\begin{equation}\label{eq:1.1}
  \xymatrix{ 0 \ar[r] & \ell_2 \ar[r] & Z_2
    \ar[r] & \ell_2\ar[r]  &0}
\end{equation}
  
\noindent defined by the quasilinear map $\mathcal{K} : \ell_2 \rightarrow \ell_2$, given on finitely supported vectors by
\[
\mathcal{K}_2\Big(\sum\limits_i x_i e_i\Big) = \sum\limits_i x_i \Big(\log\frac{\left|x_i\right|}{\|x\|_2}\Big)e_i
\]

That is, $Z_2$ is the completion of $c_{00} \times c_{00}$ with the quasinorm
\[
\|(x, y)\| = \|x - \mathcal{K}_2y\|_2 + \|y\|_2
\]

One of the interesting properties of the twisted sum \eqref{eq:1.1} is that it is singular and cosingular, i.\ e., the quotient map is strictly singular, and the inclusion is strictly cosingular. According to Kalton and Peck \cite{Kalton01}, this means that $Z_2$ is in some sense an extremal solution to Palais' problem.

In \cite{Rochberg01}, Rochberg and Weiss define the notion of `derived space' of a complex interpolation scale. This derived space happens to be an extension of the interpolation space, and the Kalton-Peck space also appears as the derived space induced by the interpolation scale $(\ell_{\infty}, \ell_1)_{\theta}$ at $\theta = \frac{1}{2}$.

The objective of this paper is to study the corresponding 3-space problem for Hilbert spaces in the category of operator spaces. Operator spaces, or quantum Banach spaces, are the noncommutative version of Banach spaces. 

One may view an operator space either as a closed subspace of the space of operators on a Hilbert space, or as a Banach space $X$ together with a sequence of norms on the spaces $M_n(X)$, $n \in \mathbb{N}$, satisfying certain axioms. Instead of operators one considers completely operators, and isomorphisms are replaced by complete isomorphisms.

Short exact sequences of operator spaces are not a new topic in the literature. They appear, for example, in \cite{Paulsen1998} and \cite{Wood1999}, but one of the main hypothesis in these works is that the short exact sequences always split. This hypothesis is motived by the interest in the module structure of the spaces. Short exact sequences of operator spaces also appear in the study of exactness and nuclearity of operator spaces in terms of tensor products (see \cite{Effros01}, chapter 14).

Of course, short exact sequences are always present when we speak of subspaces or quotients, implicitly or explicitly.

In the present work we seek to develop a theory similar to that of the theory of twisted sums of Banach spaces. Because of that, the nomenclature we use is in parallel to the Banach space scenario, despite the fact that some of the concepts are already present in the literature. We shall refer to the original nomenclature whenever we are aware of it.

In particular, our main interests are obtaining methods to find twisted sums of operator spaces and study their nontriviality. In this context, we are led to define and study completely singular and completely cosingular operators. 

Based on the Banach space case, we define completely strictly singular operators as those which are never a complete isomorphism when restricted to an infinite dimensional closed subspace of the domain, and completely strictly cosingular operators as those which when composed with a complete quotient map over an infinite dimensional operator space are never a complete quotient map.

There is an operator space which plays the role of $\ell_2$ in the category of operator spaces, the operator Hilbert space $OH$ \cite{Pisier03}, which as a Banach space is isometric to $\ell_2$. We study the following

\begin{questionPalais*} Is being completely isomorphic to $OH$ a 3-space property in the category of operator spaces which are isomorphic (isometric) to $\ell_2$, i.e., if an operator space $X$, isomorphic (isometric) as a Banach space to $\ell_2$, has a subspace completely isomorphic to $OH$ with respective quotient also completely isomorphic to $OH$, is $X$ completely isomorphic to $OH$?
\end{questionPalais*}

We show that both the isomorphic and the isometric versions of Palais' problem for operator spaces have a negative solution. We thank Gilles Pisier for pointing out how our argument for the isomorphic version could be changed to obtain an operator space isometric to $\ell_2$.

As mentioned, the Kalton-Peck space $Z_2$ may be seen as induced by complex interpolation. The operator Hilbert space $OH$ may be naturally obtained by complex interpolation of operator spaces. We study two extensions of $OH$, induced by the interpolation scales $(\min (\ell_2), \max(\ell_2))$ and $(R, C)$. We show that they induce different solutions to Palais' problem for operator spaces, which have completely strictly cosingular inclusion and completely strictly singular quotient map.

The structure of the paper is as follows: in the remainder of the introduction we give some background on operator spaces and completely bounded maps, and recall the definitions of some operator spaces with which we shall work.

In \hyperref[sec:2]{Section 2} we recall the concept of extension sequence \cite{Wood1999}, and define a complete twisted sum as the middle space of an extension sequence. Inspired by the theory of twisted sums of Banach spaces, we define different types of equivalence of extension sequences. We also define complete triviality, and completely strictly singular and completely strictly cosingular operators. We show some basic properties that follow from the definitions. We also show how the pullback and the pushout diagrams may be used to obtain extension sequences.

In \hyperref[sec:3]{Section 3} we show that, as in the Banach space case, complex interpolation induces complete twisted sums by means of the pushout diagram. We also show the duality theorem between the dual of a complete twisted sum induced by a interpolation couple and the one induced by the couple of duals. We also show some basic examples: we show that the Kalton-Peck spaces have a natural operator space structure, and how they may be used to obtain complete twisted sums of $\min(\ell_p)$ and $\max(\ell_p)$.

In \hyperref[sec:4]{Section 4} we tackle Palais' problem for operator spaces. We show that the complex interpolation scale $o(\ell_p)(\theta) = (\min(\ell_p), \max(\ell_p))_{\theta}$ induce, for $1 < p < \infty$ and $0 < \theta < 1$, extension sequences 
\begin{equation*}
\xymatrix{
0 \ar[r] & o(\ell_p)(\theta) \ar[r] & do(\ell_p)(\theta) \ar[r] & o(\ell_p)(\theta) \ar[r] & 0}
\end{equation*}
\noindent such that $o(\ell_p)(\theta)$ is complemented in $do(\ell_p)(\theta)$, but is not completely complemented. Actually, they satisfy much stronger properties (they are, as we call, completely singular and completely cosingular). In particular, taking $p = 2$ and $\theta = \frac{1}{2}$, we have an extension sequence
\begin{equation*}
\xymatrix{
0 \ar[r] & OH \ar[r] & do(\ell_2) \ar[r] & OH \ar[r] & 0}
\end{equation*}
\noindent which solves in the negative the isomorphic version of Palais' problem for operator spaces. We also show how this construction may be adapted to solve the isometric version of Palais's problem (we thank Gilles Pisier for calling our attention to this).

In sections \hyperref[sec:5]{5} and \hyperref[sec:6]{6} we analyze another solution to Palais' problem for operator spaces, this time induced by the interpolation scale $(R, C)$, obtaining an extension sequence
\[
\xymatrix{0 \ar[r] & OH \ar[r] & dOH \ar[r] & OH \ar[r] & 0}
\]

We show that this is indeed a different solution from $do(\ell_2)$, and that $do(\ell_2)$ is an operator algebra, while $dOH$ is not.

Finally, in \hyperref[sec:7]{Section 7} we state some questions that arose during this work.

\subsection{Operator Spaces}
We recall the basics of the theory of operator spaces. The reader is referred to \cite{Effros01} and \cite{Pisier04}.

We assume all spaces are over the complex field. By $M_{n}$ we mean the space of all $n\times n$-matrices with complex coefficients with the operator norm from the identification $M_{n} = B(\ell_2^n, \ell_2^n)$.

An \textit{operator space} is a Banach space together with a sequence of norms $(\|.\|_n)$ on the spaces $M_n(X)$ such that:
\begin{enumerate}
\item[O1] $\|v \oplus w\|_{m+n} = \max\{\|v\|_m, \|w\|_n\}$
\item[O2] $\|\alpha v \beta\|_n \leq \|\alpha\|\|v\|_m\|\beta\|$
\end{enumerate}

\noindent for all $v \in M_m(X)$, $w \in M_n(X)$, $\alpha \in M_{m}$, $\beta \in M_{m}$, where
\[ v \oplus w = \left( \begin{array}{ccc}
v & 0 \\
0 & w \end{array} \right)\]

Notice that if $H$ is a Hilbert space, the space of operators $B(H)$ has a natural operator space structure given by the identifications $M_n(B(H)) = B(\ell_2^n(H))$.

Given operator spaces $X$ and $Y$ and a linear map $T : X \rightarrow Y$ we have maps $T_n : M_n(X) \rightarrow M_n(Y)$ induced by $T$ by the formula
\[
T_n((x_{i, j})) = (T(x_{i, j}))
\]

The linear map $T$ is said \textit{completely bounded} if $\|T\|_{cb} = \sup_n \|T_n\| < \infty$. Clearly if $T$ is completely bounded it is bounded, and $\|T\| \leq \|T\|_{cb}$. The completely operator $T$ is a \textit{complete isometry} if each $T_n$ is an isometry, and $T$ is a \textit{complete isomorphism} if $T$ is invertible and $\|T\|_{cb}, \|T^{-1}\|_{cb} < \infty$, that is, each $T_n$ is an isomorphism with uniformly bounded constants.

Ruan's Theorem \cite{Ruan01} states that an operator space structure on $X$ is equivalent to having a completely isometric inclusion $X \subset B(H)$.

Consider $\mathcal{K}_0 = \cup_n M_n$. Then any element of the algebraic tensor product $\mathcal{K}_0 \otimes X$ may be seem as a finite matrix of elements of $X$. The space $\mathcal{K} \otimes_{min} X$ is the completion of $\mathcal{K}_0 \otimes X$ with respect to the norm $\|x\|_{min} = \|x\|_{M_n(X)}$ where $x \in M_n \otimes X \subset \mathcal{K}_0 \otimes X$.

In this context, $T : X \rightarrow Y$ is completely bounded if and only if
\[
Id_{\mathcal{K}_0} \otimes T : \mathcal{K} \otimes_{min} X \rightarrow \mathcal{K} \otimes_{min} Y
\]
\noindent is bounded, and we have $\|T\|_{cb} = \|Id_{\mathcal{K}_0} \otimes T\|$.

We have natural norm 1 operators $\pi_n : \mathcal{K}_0 \otimes X \rightarrow M_n(X)$ given by truncation, and which therefore extend to norm 1 operators from $\mathcal{K} \otimes_{min} X$ to $M_n \otimes X = M_n(X) \subset \mathcal{K} \otimes_{min} X$, that we still denote by $\pi_n$, and one proves that for every $x \in \mathcal{K} \otimes_{min} X$, $\pi_n(x) \rightarrow x$.

We have for every $x = (x_{i, j}) \in M_n(X)$ that
\begin{equation}\label{eq:1.2}
\sup \|x_{i, j}\|_X \leq \|x\|_{M_n(X)} \leq \sum\|x_{i, j}\|_X
\end{equation}

\noindent which shows that all the norms on $M_n(X)$ that come from an operator space structure are actually equivalent, and that convergence in $M_n(X)$ is convergence in each entry.

Given two operator spaces $X_1$ and $X_2$, their $\infty-$direct sum $X_1 \oplus X_2$ is equipped with the operator space structure defined by
\[
\|(x_1, x_2)\| = \max\{\|x_1\|, \|x_2\|\}
\]
\noindent for $(x_1, x_2) \in M_n(X_1 \oplus X_2) = M_n(X_1) \oplus M_n(X_2)$.

Also, if $X$ is an operator space and $Y$ is a closed subspace of $X$, then $Y$ is naturally equipped with an operator space structure, simply by restricting the norms of $M_n(X)$ to $M_n(Y)$. The quotient operator space $X/Y$ is defined by the isometric identifications
\[
M_n(X/Y) = M_n(X)/M_n(Y)
\]

Every Banach space $X$ has a minimal and a maximal operator space structure, respectively denoted by $\min(X)$ and $\max(X)$ \cite{Blecher1991}. This minimality/maximality is expressed by the fact that if $X$ is endowed with an operator space structure, then we have c.b. norm $1$ inclusions
\[
\max(X) \subset X \subset \min(X)
\]

We shall use the following characterizations for $\min(X)$:
\begin{equation}\label{eq:1.3}
    \|x\|_{M_n(\min(X))} = \sup\{\|\sum \lambda_i \mu_j x_{i,j}\|_X : \sum \left|\lambda_i\right|^2 \leq 1, \sum \left|\mu_j\right|^2 \leq 1\}
\end{equation}

\begin{equation}\label{eq:1.4}
    \|x\|_{M_n(\min(X))} = \sup\{\|f_n(x)\|_{M_n} : f \in X^*, \|f\| \leq 1\}
\end{equation}

We shall also use the fact that
\begin{equation}\label{eq:1.5}
    \|x\|_{M_n(\max(X))} = \inf\{\|A\|\|D\|\|B\|\}
\end{equation}

\noindent where the infimum is over all decompositions $x = ADB$, with $A$ and $B$ scalar matrices and $D$ a diagonal matrix with entries in $X$ \cite{Paulsen1992}.

Let $X$ and $Y$ be operator spaces, and let $CB(X, Y)$ be the space of completely operators from $X$ to $Y$ with the c.b. norm. The operator space structure of the dual of $X$ is given by the natural identification $M_n(X^*) = CB(X, M_n)$.

In particular, we have completely isometric identifications $\min(X^*) = (\max(X))^*$ and $\max(X^*) = (\min(X))^*$.

Also, if $T : X \rightarrow Y$ we have
\[
\|T : \max(X) \rightarrow Y\|_{cb} = \|T : X \rightarrow \min(Y)\|_{cb} = \|T\|
\]

An operator space $X$ is called \textit{homogeneous} if for every operator $T : X \rightarrow X$, we have $\|T\|_{cb} = \|T\|$. So $\min(X)$ and $\max(X)$ are homogeneous operator spaces, that is, if $T : X \rightarrow X$, we have:
\begin{equation}\label{eq:1.6}
\|T : \min(X) \rightarrow \min(X)\|_{cb} = \|T : \max(X) \rightarrow \max(X)\|_{cb} = \|T\|
\end{equation}

If $X$ is infinite dimensional then $\min(X)$ and $\max(X)$ are not completely isomorphic.

We will also consider the row and column operator space structures on $\ell_2$, denoted respectively by $R$ and $C$, which are given by
\begin{eqnarray*}
\|\sum x_k \otimes e_k\|_{M_n(R)} & = & \|\sum x_k x_k^*\|^{\frac{1}{2}} \\
\|\sum x_k \otimes e_k\|_{M_n(C)} & = & \|\sum x_k^* x_k\|^{\frac{1}{2}}
\end{eqnarray*}

\noindent where $(e_k)_{k \in \mathbb{N}}$ is any orthonormal basis for $\ell_2$. These are also homogeneous operator spaces.

\section{Complete twisted sums}\label{sec:2}
\subsection{Basic definitions and results}Suppose we have an exact sequence of Banach spaces
  \[
  \xymatrix{ 0 \ar[r] & Y \ar[r]^i & X
    \ar[r]^q & Z\ar[r]  &0}
  \]

\noindent where the arrows are bounded linear maps. It is a consequence of the open mapping theorem that $i$ is an isomorphism onto its image and that $q$ induces an isomorphism between $X/Y$ and $Z$.

As noted in \cite{Wood1999}, the situation is different when we deal with short exact sequences of operator spaces. Let $X$ be an infinite dimensional Banach space and consider the following exact sequences in the category of operator spaces:
\[
  \xymatrix{ 0 \ar[r] & \max(X) \ar[r]^{Id_X} & \min(X)
    \ar[r] & 0\ar[r]  &0}
  \]
  
\noindent and

  \[
  \xymatrix{ 0 \ar[r] & 0 \ar[r] & \max(X)
    \ar[r]^{Id_X} & \min(X)\ar[r]  &0}
  \]
  
Then the arrows are completely operators, but $Id_X$ is not an isomorphism. This motivates the following definition \cite{Wood1999}:
\begin{definition}\label{def:2.1}
An \textit{extension sequence} of operator spaces is a short exact sequence
\[
  \xymatrix{ 0 \ar[r] & Y \ar[r]^i & X
    \ar[r]^{q} & Z\ar[r]  &0}
  \]
where $Y$, $X$ and $Z$ are operator spaces, the arrows are completely operators, $i$ is a complete isomorphism onto its image, and the completely operator induced by $q$ is a complete isomorphism.
\end{definition}

We will also say that $X$ is a \textit{complete twisted sum} of $Y$ and $Z$. If $Y = Z$, we will say that $X$ is a \textit{complete extension} of $Y$.

We recall that a \textit{complete surjection} $q : X \rightarrow Z$ between operator spaces is a completely bounded surjective map such that the induced operator from $X/\ker(q)$ onto $Z$ is a complete isomorphism. That is, in the definition of an extension sequence, the quotient is a complete surjection. We have that this happens for $q$ completely bounded if and only if
\[
Id_{\mathcal{K}_0} \otimes q : \mathcal{K} \otimes_{\min} X \rightarrow \mathcal{K} \otimes_{\min} Z
\]

\noindent is a surjection \cite{Pisier04}. Also, $i : Y \rightarrow X$ is a complete isomorphism if and only if
\[
Id_{\mathcal{K}_0} \otimes i : \mathcal{K} \otimes_{\min} Y \rightarrow \mathcal{K} \otimes_{\min} X
\]

\noindent is an isomorphism \cite{Pisier04}. This already allows us to prove some basic facts about extension sequences.

\begin{props}\label{prop:2.2}
Every extension sequence
\[
  \xymatrix{ 0 \ar[r] & Y \ar[r]^i & X
    \ar[r]^{q} & Z\ar[r]  &0}
  \]
induces a short exact sequence of Banach spaces
\[
  \xymatrix{ 0 \ar[r] & \mathcal{K}\otimes_{\min}Y \ar[r]^{Id_{\mathcal{K}_0}\otimes i} & \mathcal{K}\otimes_{\min}X
    \ar[r]^{Id_{\mathcal{K}_0}\otimes q} & \mathcal{K}\otimes_{\min}Z\ar[r]  &0}
  \]
\end{props}
\begin{proof}
That $Id_{\mathcal{K}_0} \otimes i$ is injective follows from it being an isomorphism, and $Id_{\mathcal{K}_0} \otimes q$ is a surjection since $q$ is a complete surjection. So we need to prove that the image of $Id_{\mathcal{K}_0} \otimes i$ coincides with the kernel of $Id_{\mathcal{K}_0} \otimes q$.

Let $y \in \mathcal{K} \otimes_{\min} Y$ and take $y_n \in \mathcal{K}_0 \otimes_{\min} Y$ converging to $y$. Then
\begin{equation*}
Id_{\mathcal{K}_0}\otimes q (Id_{\mathcal{K}_0} \otimes i (y))  =  \lim Id_{\mathcal{K}_0} \otimes q (Id_{\mathcal{K}_0} \otimes i(y_n)) 
 =  0
\end{equation*}

Indeed, if $y_n \in M_k(Y)$, then $Id_{\mathcal{K}_0} \otimes q (Id_{\mathcal{K}_0} \otimes i(y_n)) = q_k i_k (y_n) = 0$.

To prove the reverse inclusion, let $x$ be in the kernel of $Id_{\mathcal{K}_0} \otimes q$. For every $n$ it is easy to see that $Id_{\mathcal{K}_0} \otimes q$ and $\pi_n$ commute on $\mathcal{K}_0 \otimes X$. Therefore $Id_{\mathcal{K}_0} \otimes q$ and $\pi_n$ commute on $\mathcal{K} \otimes_{\min} X$, and
\begin{equation*}
Id_{\mathcal{K}_0} \otimes q(\pi_n(x))  =  \pi_n (Id_{\mathcal{K}_0} \otimes q(x))  =  0
\end{equation*}

By the exactness of the original sequence, there is $y_n \in \mathcal{K}_0 \otimes Y$ such that $\pi_n(x) = Id_{\mathcal{K}_0} \otimes i(y_n)$. Since $i$ is a complete isomorphism, there is a constant $C$ such that, for every $n$ and $m$, we have
\begin{equation*}
\|y_n - y_m\|  \leq  C\|Id_{\mathcal{K}_0} \otimes i (y_n - y_m)\| 
               =  C \|\pi_n(x) - \pi_m(x)\|
\end{equation*}

So $(y_n)$ is a Cauchy sequence, and there is $y \in \mathcal{K} \otimes_{\min} Y$ such that $(y_n)$ converges to $y$. Then
\begin{equation*}
Id_{\mathcal{K}_0} \otimes i (y)  =  \lim Id_{\mathcal{K}_0} \otimes i (y_n)
     =  \lim \pi_n(x)
     =  x
\end{equation*}

Therefore, $\ker Id_{\mathcal{K}_0} \otimes q = \Ima Id_{\mathcal{K}_0} \otimes i$.
\end{proof}

It is then natural that much of the behaviour of twisted sums of Banach spaces will be reproduced in the noncommutative case. For example, from the last proposition we can also reclaim the classical 3-lemma to the operator space setting.

\begin{props}[3-lemma]\label{prop:2.3}
Suppose we have a commutative diagram
  \[
  \xymatrix{ 0 \ar[r] & A \ar[r]\ar[d]^{\alpha} & B
    \ar[r]\ar[d]^{\beta} & C\ar[r] \ar[d]^{\gamma} &0 \\ 0 \ar[r] & D
    \ar[r] & E \ar[r]& F\ar[r] &0}
  \]
where the rows are extension sequences and $\alpha$, $\beta$ and $\gamma$ are completely operators. If $\alpha$ and $\gamma$ are complete surjections (complete isomorphisms), so is $\beta$.
\end{props}
\begin{proof}
One just has to look at the induced diagram in the Banach space setting and use the classical 3-lemma (\cite{Castillo02}, page 3).
\end{proof}

The following definition is inspired by the Banach space setting.

\begin{definition}\label{def:2.4}
Suppose $X_i$ is a complete twisted sum of $Y_i$ and $Z_i$, $i = 1, 2$, and that we have a commutative diagram
  \[
  \xymatrix{ 0 \ar[r] & Y_1 \ar[r]\ar[d]^{\alpha} & X_1
    \ar[r]\ar[d]^{\beta} & Z_1\ar[r] \ar[d]^{\gamma} &0 \\ 0 \ar[r] & Y_2
    \ar[r] & X_2 \ar[r]& Z_2\ar[r] &0}
  \]
  
The complete twisted sums $X_1$ and $X_2$ are
\begin{enumerate}
\item \textit{completely isomorphically equivalent}, if $\alpha$, $\beta$ and $\gamma$ are complete isomorphisms;
\end{enumerate}

If $Y_1 = Y_2$ and $Z_1 = Z_2$,
\begin{enumerate}
\item[(2)] \textit{completely projectively equivalent}, if $\alpha$ and $\gamma$ are multiples of the identity and $\beta$ is a complete isomorphism;
\item[(3)] \textit{completely equivalent}, if $\alpha = Id_{Y_1}$, $\gamma = Id_{Z_1}$, and $\beta$ is a complete isomorphism.
\end{enumerate}
\end{definition}

Under the hypothesis that we are working with extension sequences, the 3-lemma shows that complete equivalence is the same as the equivalence of 1-extensions of \cite{Paulsen1998}.

We recall that twisted sums of Banach spaces for which the quasi-norm is equivalent to a norm come defined by \textit{0-linear maps} $F : Z \rightarrow Y$, i.e., homogeneous maps for which there is a constant $K > 0$ such that
\begin{equation*}
    \|\sum F(z_i)\| \leq K \sum\|z_i\|
\end{equation*}
whenever $\sum z_i =0$ \cite{Castillo02}. That is, given a 0-linear map $F : Z \rightarrow Y$ we can endow the product $Y \times Z$ with the quasi-norm $\|(y, z)\|_F = \|y - Fz\| + \|z\|$ which is equivalent to a norm, and we have that $Y \cong \{(y, 0) : y \in Y\}$ by the inclusion $i(y) = (y,0)$, and the quotient map $q(y, z) = z$ induces an isomorphism $Z \cong (Y \oplus_F Z)/F$. 

Conversely, given a twisted sum $X$ of $Y$ and $Z$ such that the quasi-norm on $X$ is equivalent to a norm, we can get a 0-linear map $F : Z \rightarrow Y$ such that the twisted sum $Y \oplus_F Z$ is equivalent to $X$. 

It follows that $Y \oplus_F Z$ is equivalent to $Y \oplus_G Z$ if and only if there is a linear map $l : Z \rightarrow Y$ such that
\begin{equation*}
    \|F - G - l\| =  \sup\limits_{\|z\| = 1} \|F(z) - G(z) - l(z)\| < \infty
\end{equation*}
\noindent (see \cite{Kalton01}).

A twisted sum $X$ of $Y$ and $Z$ is said \textit{trivial} if the image of $Y$ under the inclusion is complemented in $X$, and for $Y \oplus_F Z$ this is equivalent to $F$ being at finite distance from a linear map.

Suppose we have a (Banach) twisted sum
\[
  \xymatrix{ 0 \ar[r] & Y \ar[r]^{i} & X
    \ar[r]^q & Z\ar[r]  &0}
  \]

\noindent and suppose $Y$, $X$ and $Z$ are operator spaces. Then we have an induced twisted sum
\[
  \xymatrix{ 0 \ar[r] & M_n(Y) \ar[r]^{i_n} & M_n(X)
    \ar[r]^{q_n} & M_n(Z)\ar[r]  &0}
  \]
  
\noindent for every $n \in \mathbb{N}$. The following proposition shows the regularity of this construction.

\begin{props}\label{prop:2.5}
Suppose we have a twisted sum of Banach spaces
\begin{equation}\label{eq:2.1}
  \xymatrix{ 0 \ar[r] & Y \ar[r]^{i} & X
    \ar[r]^q & Z\ar[r]  &0}
\end{equation}
\noindent and that $Y$, $X$ and $Z$ are operator spaces. Then, for every $n \in \mathbb{N}$, \eqref{eq:2.1} is trivial if and only if the twisted sum
\[
  \xymatrix{ 0 \ar[r] & M_n(Y) \ar[r]^{i_n} & M_n(X)
    \ar[r]^{q_n} & M_n(Z)\ar[r]  &0}
  \]
\noindent is trivial.
\end{props}
\begin{proof}
Let $F : Z \rightarrow Y$ be the quasilinear map that defines the twisted sum, i.e., there is a bounded linear map $T : X \rightarrow Y \oplus_F Z$ making the following diagram commute:
  \[
  \xymatrix{ 0 \ar[r] & Y \ar[r]\ar@{=}[d] & X
    \ar[r]\ar[d]^{T} & Z\ar[r] \ar@{=}[d] &0 \\ 0 \ar[r] & Y
    \ar[r] & Y \oplus_F Z \ar[r]& Z\ar[r] &0}
  \]
  
Then, by \eqref{eq:1.2}, for every $n$, $F_n : M_n(Z) \rightarrow M_n(Y)$ defined by $(z_{i, j}) \mapsto (F(z_{i, j}))$ is a quasilinear map and we have a commutative diagram
  \[
  \xymatrix{ 0 \ar[r] & M_n(Y) \ar[r]\ar@{=}[d] & M_n(X)
    \ar[r]\ar[d]^{T_n} & M_n(Z)\ar[r] \ar@{=}[d] &0 \\ 0 \ar[r] & M_n(Y)
    \ar[r] & M_n(Y) \oplus_{F_n} M_n(Z) \ar[r]& M_n(Z)\ar[r] &0}
  \]
  
This means that $F_n$ defines the twisted sum $M_n(X)$ for every $n$.

The only if part is simple. Let us show the if part.

If $M_n(X)$ is trivial, there is $A : M_n(Z) \rightarrow M_n(Y)$ linear such that $F_n - A$ is bounded.

Consider $j : Z \rightarrow M_n(Z)$ the inclusion in the first entry and $\pi : M_n(Y) \rightarrow Y$ the projection of the first entry, and let $\tilde{A} = \pi \circ A \circ j : Z \rightarrow Y$. Then $\tilde{A}$ is linear and
\begin{equation*}
    \|F - \tilde{A}\| = \|\pi \circ F_n \circ j - \pi \circ A \circ j\| \leq \|F_n - A\| < \infty
\end{equation*}

Therefore $F$ is trivial.
\end{proof}

Notice that there is no reason for $T$ being a completely operator, as we shall see in \hyperref[sec:4]{Section 4}.

Given two operator spaces $Y$ and $Z$, we have their \textit{trivial complete twisted sum}, given by the extension sequence:
\[
  \xymatrix{ 0 \ar[r] & Y \ar[r]^{i} & Y \oplus Z
    \ar[r]^q & Z\ar[r]  &0}
\]
\noindent where $i$ is inclusion in the first coordinate and $q$ is the projection of the second coordinate.

\begin{definition}\label{def:2.6}
Let $Y$ and $Z$ be operator spaces. A complete twisted sum $X$ of $Y$ and $Z$ will be called \textit{completely trivial} if it is completely equivalent to $Y \oplus Z$.
\end{definition}

In the same manner, we will talk of completely trivial extension sequences. In the language of \cite{Paulsen1998}, the sequence is $\mathbb{C}-\mathbb{C}-$split, and in the language of \cite{Wood1999}, it is admissible.

It is clear that if a complete twisted sum is completely trivial, it is trivial as a Banach space twisted sum. As we will see, the converse is not true.

We recall that if $T : X \rightarrow Y$ is a linear map, a \textit{section} of $T$ is a linear map $s : Y \rightarrow X$ such that $T \circ s = Id_Y$ and a \textit{retraction} of $T$ is a linear map $r : Y \rightarrow X$ such that $r \circ T = Id_X$.

The proof of the following proposition is similar to the proof in the Banach space case, and may be essentially found in \cite{Wood1999} (Lemma 3.2.6).
\begin{props}\label{prop:2.7}
Let
\[
  \xymatrix{ 0 \ar[r] & Y \ar[r]^{i} & X
    \ar[r]^q & Z\ar[r]  &0}
  \]
  
\noindent be a complete twisted sum. The following are equivalent:
\begin{enumerate}
    \item [a)] It is completely trivial.
    \item [b)] There is a completely bounded section of q.
    \item [c)] There is a completely bounded retraction of i.
\end{enumerate}
\end{props}

This allows us to get some cases where triviality and complete triviality are equivalent.

\begin{props}\label{prop:2.8} 
For the following extension sequences, complete triviality and triviality as a Banach spaces twisted sum are equivalent:
\begin{enumerate}
    \item[a)]    \[
    \xymatrix{ 0 \ar[r] & Y \ar[r] & X
    \ar[r] & \max(Z)\ar[r]  &0}
    \]
    
    \item[b)] \[
    \xymatrix{ 0 \ar[r] & \min(Y) \ar[r] & X
    \ar[r] & Z\ar[r]  &0}
    \]
\end{enumerate}
\end{props}
\begin{proof}
One only has to notice that, for the first sequence, any bounded section for the quotient map is completely bounded, and in the second any bounded retraction of the inclusion is completely bounded.
\end{proof}

\subsection{Singularity and cosingularity}\label{sec:2.2}
Suppose we have a twisted sum of Banach spaces
\[
    \xymatrix{ 0 \ar[r] & Y \ar[r]^i & X
    \ar[r]^q & Z\ar[r]  &0}
    \]

\noindent and let $W$ be a subspace of $Z$ (subspaces are assumed closed, unless otherwise stated). We can then induce a twisted sum
\[
    \xymatrix{ 0 \ar[r] & Y \ar[r]^{i} & q^{-1}(W)
    \ar[r]^p & W\ar[r]  &0}
    \]
where $p = q|_{q^{-1}(W)}$. The twisted sum is called \textit{singular} if for each $W$ infinite dimensional the induced twisted sum is nontrivial. This happens if and only if the quotient map $q$ is \textit{strictly} singular, i.e., it is not an isomorphism when restricted to any closed infinite dimensional subspace of $X$ (see \cite{Castillo03}, for example).

Analogously, we have
\begin{props}\label{prop:2.9}
Consider an extension sequence
\[
    \xymatrix{ 0 \ar[r] & Y \ar[r]^i & X
    \ar[r]^q & Z\ar[r]  &0}
    \]
There exists an infinite dimensional closed subspace of $X$ such that $q$ restricted to this subspace is a complete isomorphism if and only if there is an infinite dimensional closed subspace of $Z$ such that the induced complete twisted sum is completely trivial.
\end{props}

So we are led to the following definition:

\begin{definition}\label{def:2.10}
Let $X$ and $Z$ be operator spaces and let $T : X \rightarrow Z$ be a completely operator. We say that $T$ is \textit{completely strictly singular} (c.s.s.) if $T$ is not a complete isomorphism when restricted to any infinite dimensional closed subspace of $X$.

A complete twisted sum (or extension sequence) will be called \textit{completely singular} if the quotient map is a c.s.s. operator.
\end{definition}

A word of caution is needed: $T$ being c.s.s. is \textit{not} equivalent to $Id_{\mathcal{K}_0} \otimes T$ being strictly singular, as the next example shows.

\begin{example}\label{ex:2.11}
Consider $Id_{\ell_2} : \max(\ell_2) \rightarrow \min(\ell_2)$. It is completely bounded, is an isometry, but is not a complete isomorphism. Actually, it is c.s.s. Indeed, if $W$ is any closed infinite dimensional subspace of $\max(\ell_2)$, then it is completely isometric to $\max(\ell_2)$ \cite{Vinod01}, and its image is completely isometric to $\min(\ell_2)$. Therefore $Id_{\ell_2}|_W$ is not a complete isomorphism.

However, by \eqref{eq:1.2}, for every $n \in \mathbb{N}$, $Id_{\mathcal{K}_0} \otimes Id_{\ell_2}$ is an isomorphism when restricted to $M_n \otimes \ell_2$.
\end{example}

This example also shows that while singularity clearly implies complete singularity, even an isometry might be completely singular.

Twisted sums of Banach spaces in which the inclusion is strictly cosingular also appear in the literature, for example, the Kalton-Peck spaces \cite{Kalton01}.

\begin{definition}\label{def:2.12}
A completely operator $T : Y \rightarrow X$ between operator spaces is \textit{completely strictly cosingular} (c.s.c.) if whenever there is $q : X \rightarrow E$ such that $q$ and $q \circ T : Y \rightarrow E$ are complete quotient maps, it follows that $E$ is finite dimensional.

A complete twisted sum (or extension sequence) will be called \textit{completely cosingular} if the inclusion is completely strictly cosingular.
\end{definition}

\begin{example}\label{ex:2.13} Again, strict cosingularity implies complete strict cosingularity, but the converse is not true. By \hyperref[ex:2.11]{Example 2.11} and part (2) of \hyperref[prop:2.15]{Proposition 2.15}, $Id_{\ell_2} : \max(\ell_2) \rightarrow \min(\ell_2)$ is c.s.c., despite being an onto isometry. Also, $Id_{\mathcal{K}_0} \otimes Id_{\ell_2}$ is not strictly cosingular, since if $P_{11}$ is the projection in the first entry, we can take $q = P_{11} \otimes Id_{\ell_2}$ in the definition of strict cosingularity.
\end{example}

We shall denote the class of completely strictly singular operators from $X$ into $Z$ by $CSS(X, Z)$ and the class of completely strictly cosingular operators from $Y$ into $X$ by $CSC(Y, X)$. The following properties may be proved as in the Banach space case (see \cite{Shannon01}, for example), with the particularity that we cannot use the open mapping theorem. A proof that uses the open mapping theorem may be done by considering the diagram induced by the minimal tensor product with $\mathcal{K}_0$.

\begin{props}\label{prop:2.14}
The classes $CSS(X, Z)$ and $CSC(Y, X)$ are closed by composition with completely operators on the left and on the right.
\end{props}
\begin{proof}
Let $T \in CSS(X, Z)$ and $S \in CB(X', X)$. We will prove that $T\circ S \in CSS(X', Z)$.

Suppose that $W$ is a closed subspace of $X'$ such that $(T\circ S)|_W$ is a complete isomorphism onto its image. Then, there is $C>0$ such that for all $w \in M_n(W)$, $n \in \mathbb{N}$,
\[
\|S_n(w)\| \leq \|S\|_{cb}\|w\| \leq C\|S\|_{cb}\|T_n \circ S_n (w)\| \leq C\|T\|_{cb}\|S\|_{cb}\|S_n(w)\|
\]

\noindent which shows that $T|_{S(W)} : S(W) \rightarrow Z$ is a complete isomorphism onto its image, and since $T(S(W))$ is closed, so is $S(W)$. Since $T$ is c.s.s., this means that $S(W)$ is finite dimensional.

But $(T\circ S)|_W$ being a complete isomorphism implies that $S|_W$ is injective, and therefore $W$ is finite dimensional, which proves that $T \circ S \in CSS(X', Z)$.

Now let $S \in CB(Z, Z')$. We will prove that $S \circ T \in CSS(X, Z')$.

Let $W$ be a closed subspace of $X$ such that $(S \circ T)|_W$ is a complete isomorphism. This implies that there is a constant $C > 0$ such that for every $w \in M_n(W)$, $n \in \mathbb{N}$,
\[
\|w\| \leq C \|S_n \circ T_n(w)\| \leq C\|S\|_{cb}\|T_n(w)\|
\]

\noindent which implies that $T$ is a complete isomorphism on $W$, and so $W$ is finite dimensional. Therefore, $S \circ T \in CSS(X, Z')$.

Now, suppose that $T \in CSC(Y, X)$, and that $S \in CB(Y', Y)$. We show that $T \circ S \in CSC(Y', X)$. 

Let $q : X \rightarrow E$ and $p' : Y' \rightarrow E$ be complete quotient maps such that $q \circ T \circ S = p'$.

Let $p = q \circ T : Y \rightarrow E$. We show that it is a complete quotient map. Surjectivity follows from $p \circ S$ being surjective. Also, it is clear that it is a c.b. map. We must prove that there is a constant $C > 0$ such that for every $y \in M_n(Y)$, $n \in \mathbb{N}$,
\[
\|y + \ker p_n\| \leq C \|p_n(y)\|
\]

Given $y \in M_n(Y)$, there is $y' \in M_n(Y')$ such that $p'_n(y') = p_n(y)$. But $p' = p \circ S$. Therefore, since $p'$ is a complete quotient map, there is $K > 0$ such that:
\begin{eqnarray*}
\|y + \ker p_n\| & = & \|S_n(y') + \ker p_n\| \\
    & = & \inf\{\|S_n(y') + \tilde{y}\| : p_n(\tilde{y}) = 0\} \\
    & \leq & \inf\{\|S_n(y') + S_n(y'')\| : p'_n(y'') = 0\} \\
    & \leq & \|S\|_{cb} \inf\{\|y' + y''\| : p'_n(y'') = 0\} \\
    & \leq & K\|S\|_{cb} \|p'_n(y')\| \\
    & = & K \|S\|_{cb} \|p_n(y)\|
\end{eqnarray*}

So $p$ is a complete quotient, and $q \circ T = p$. Since $T$ is c.s.c., $E$ is finite dimensional, and $T \circ S \in CSC(Y', X)$.

Finally, let $T \in CSC(Y, X)$, and $S \in CB(X, X')$. We will prove that $S \circ T \in CSC(Y, X')$.

Let $q' : X' \rightarrow E$ and $p : Y \rightarrow E$ be complete quotient maps such that $q' \circ S \circ T = p$. Let $q = q' \circ S : X \rightarrow E$. We prove that $q$ is a complete quotient map.

It is surjective, since $p = q \circ T$. We must show that there is $C > 0$ such that for all $x \in M_n(X)$, $n \in \mathbb{N}$,
\[
\|x + \ker q_n\| \leq C \|q_n(x)\|
\]

Since $q \circ T = p$, and $p$ is surjective, there is $y \in M_n(Y)$ such that $q_n (T_n(y)) = p_n(y) = q_n(x)$. Since $p$ is a complete quotient, there is $K > 0$ such that
\begin{eqnarray*}
\|x + \ker q_n\| & = & \|T_n(y) + \ker q_n\| \\
& = & \inf\{\|T_n(y) + \tilde{x}\| : q_n(\tilde{x}) = 0\} \\
    & \leq & \inf\{\|T_n(y) + T_n(y')\| : q_n(T_n(y')) = 0\} \\
    & \leq & \|T\|_{cb} \inf\{\|y + y'\| : p_n(y') = 0\} \\
    & \leq & K \|T\|_{cb} \|p_n(y)\| \\
    & = & K \|T\|_{cb} \|q_n(x)\|
\end{eqnarray*}

So $q$ is a complete quotient, and $q \circ T = p$. Since $T \in CSC(Y, X)$, $E$ is finite dimensional, and $S \circ T \in CSC(Y, X')$.
\end{proof}

In particular, complete singularity and complete cosingularity are preserved by the different equivalences of complete twisted sums.

\begin{props}\label{prop:2.15} Let $X$ and $Y$ be operator spaces, and $T \in CB(X, Y)$.
\begin{enumerate}
\item If $T^* \in CSC(Y^*, X^*)$, then $T \in CSS(X, Y)$.
\item If $T^* \in CSS(Y^*, X^*)$, then $T \in CSC(X, Y)$.
\end{enumerate}
\end{props}
\begin{proof}
We will use the fact that $i$ is a complete isomorphism onto its image if and only if $i^*$ is a complete quotient map, and that $q$ is a complete quotient map if and only if $q^*$ is a complete isomorphism onto its image \cite{Pisier04}.

(1) Suppose that $W$ is a closed subspace of $X$ such that $T|_W$ is a complete isomorphism. Let $i : W \rightarrow X$ be the inclusion. Then we have that $T \circ i = j$ is a complete isomorphic injection.

Taking adjoints, we have $i^* \circ T^* = j^*$, where $i^* : X^* \rightarrow W^*$ and $j^* : Y^* \rightarrow W^*$ are complete quotient maps. Since $T^* \in CSC(Y^*, X^*)$, $W^*$ is finite dimensional, and so is $W$. Therefore, $T \in CSS(X, Y)$.

(2) Let $q : Y \rightarrow E$ and $p : X \rightarrow E$ be complete quotient maps such that $q \circ T = p$.

Taking adjoints, we have $T^* \circ q^* = p^*$, where $q^* : E^* \rightarrow Y^*$ and $p^* : E^* \rightarrow X^*$ are complete isomorphic injections, that is, $T^*|_{q^*(E^*)}$ is a complete isomorphism, and since $q^*(E^*)$ is closed in $Y^*$ and $T^* \in CSS(Y^*, X^*)$, $q^*(E^*)$ is finite dimensional, and so is $E$. Therefore, $T \in CSC(X, Y)$.
\end{proof}

\subsection{Pullback}\label{sec:2.4}
We now turn our attention to the question of how to obtain complete twisted sums. To this end, we will study the pullback and pushout sequences, adapted from the Banach space case, as presented in \cite{Castillo02}. See also \cite{Paulsen1998} for a presentation in the more general context of modules over operator algebras, again under the hypothesis of splitting.

Let $X$, $Y$ and $Z$ be Banach spaces, and let $A : X \rightarrow Z$ and $V : Y \rightarrow Z$ be operators. Then there is a space $\Xi$ with operators $\pi_X$ and $\pi_Y$ making the following diagram commute:
\[
\xymatrix{
\Xi \ar[d]^{\pi_X} \ar[r]^{\pi_Y} &Y\ar[d]^B\\
X \ar[r]^A &Z}
\]

By definition, $\Xi = \{(x, y) \in X \oplus Y : Ax = By\}$, with the norm induced from $X \oplus Y$, and $\pi_X$ and $\pi_Y$ are the natural projections. If $A$ and $B$ are bounded, we have that $\Xi$ is a closed subspace of $X \oplus Y$. The space $\Xi$ is called the \textit{pullback} of $\{A, B\}$.

If $X$, $Y$ and $Z$ are operator spaces, we have that $\Xi$ is an operator space, and $\|\pi_X\|_{cb} = \|\pi_Y\|_{cb} = 1$.

Suppose we have a diagram of Banach spaces and operators
\[
\xymatrix{
0 \ar[r] & Y \ar[r]^i & X \ar[r]^q & Z \ar[r] &0\\
&&& V \ar[u]^Q}
\]

\noindent where the first line is exact, $Q$ is surjective, and the spaces have an operator space structure. We can then consider the pullback $\{q, Q\}$, and we obtain a commutative diagram \cite{Sanchez02}:
\[
\xymatrix{
& 0 & 0 & 0\\
0 \ar[r] & Y \ar@{=}[d]\ar[r]^i\ar[u] & X \ar[r]^q\ar[u] & Z \ar[r]\ar[u] & 0\\
0 \ar[r] & Y \ar[r]^j & \Xi \ar[r]^{\pi_V}\ar[u]^{\pi_X} & V \ar[u]^Q\ar[r] & 0 \\
0 \ar[r] & 0 \ar[r]\ar[u] & K \ar[r]\ar[u]^k & \ker(Q) \ar[u]\ar[r] & 0\\
& 0 \ar[u] & 0 \ar[u] & 0 \ar[u]
}
\]

\noindent where $K = \ker(\pi_X)$ is given the operator space structure induced by $\Xi$.

\begin{props}\label{prop:2.16} Suppose we have a diagram induced by pullback like above.
\begin{enumerate}
\item[a)] The third line is an extension sequence, with $K$ completely isometric to $\ker(Q)$.
\item[b)] If $Q$ is c.b. and the first line is an extension sequence, then the second line is an extension sequence.
\item[c)] If $q$ is c.b. and the third column is an extension sequence, then the second column is an extension sequence.
\end{enumerate}
\end{props}
\begin{proof}
$a)$ We have $K = \ker(\pi_X) = \{(0, v) \in X \oplus V : Q(v) = 0\} = 0 \oplus \ker(Q)$.

$b)$ First, we show that $j$ is a complete isomorphic injection. It is defined by $j(y) = (i(y), 0)$, $y \in Y$. We have $j(Y) = \{(i(y), 0) \in X \oplus V : y \in Y\} = i(Y) \oplus 0$. Since $i$ is a complete isomorphism onto its image, it follows that the same is true for $j$.

We have to prove that the operator induced by $\pi_V$ is a complete isomorphism. Let us call it $\tilde{\pi_V}$.

Since $\|\pi_V\|_{cb} = 1$, we have $\|\tilde{\pi_V}\|_{cb} \leq 1$. 

Now let $(u, v) \in M_n(\Xi)$. We have $(u, v) + M_n(j(Y)) = \{(u + i_n(y), v) \in M_n(\Xi) : y \in M_n(Y)\}$. 

Take $C > 0$ such that $\|x + i_n(M_n(Y))\| \leq C\|q_n(x)\|$ for every $n$ and for every $x \in M_n(X)$. Then:
\begin{eqnarray*}
\|(u, v) + M_n(j(Y))\| & = & \inf\{\|(u + i_n(y), v)\| : y \in M_n(Y)\} \\
                    & \leq & \inf\{\|u + i_n(y)\| : y \in M_n(Y)\} + \|v\| \\
                    & \leq & C\|q_n(u)\| + \|v\| \\
                    & = & C\|Q_n(v)\| + \|v\| \\
                    & \leq & (1 + C\|Q\|_{cb})\|v\| \\
                    & = & (1 + C\|Q\|_{cb})\|(\tilde{\pi_V})_n ((u, v) + M_n(j(Y)))\|
\end{eqnarray*}

Affirmation $c)$ is proved like $b)$.
\end{proof}

\subsection{Pushout}\label{sec:2.4}
Let $Y$, $X$ and $K$ be Banach spaces, and let $A : Y \rightarrow X$ and $B : Y \rightarrow K$ be operators. There is a space $\mathcal{PO}$ and two operators $j_K$ and $j_X$ making the following diagram commute:
\[
\xymatrix{
Y \ar[d]^{B} \ar[r]^{A} &X\ar[d]^{j_X}\\
K \ar[r]^{j_K} &\mathcal{PO}}
\]

Let $D = \{(Ay, -By) \in X \oplus K : y \in Y\}$, and $\Delta = \overline{D}$. Then $\mathcal{PO} = (X\oplus K)/\Delta$, $j_K(k) = (0, k) + \Delta$, $j_X(x) = (x, 0) + \Delta$. The space $\mathcal{PO}$ is called the \textit{pushout} of $\{A, B\}$.

If $X$, $Y$ and $K$ are operator spaces, so is $\mathcal{PO}$, and $j_K$ and $j_X$ are completely bounded (take $y = 0$).

Note that if $A$ or $B$ is an isomorphism, then $D = \Delta$. 

\begin{props}\label{prop:2.17}
Suppose we have a diagram
\[
\xymatrix{
0 \ar[r] & Y \ar[r]^{i}\ar[d]^T & X \ar[r]^q & Z \ar[r] & 0 \\
& K}
\]

\noindent where $Y$, $X$ and $Z$ are operator spaces, the line is an extension sequence and $T$ is c.b. Then the pushout of $\{i, T\}$ induces a commutative diagram
\[
\xymatrix{
0 \ar[r] & Y \ar[r]^{i}\ar[d]^T & X \ar[d]^{\tilde{T}}\ar[r]^q & Z \ar@{=}[d]\ar[r] & 0 \\
0 \ar[r] & K \ar[r]^j & \mathcal{PO} \ar[r]^p & Z \ar[r] & 0}
\]

\noindent where $p$ is defined by $p((x, k) + \Delta) = q(x)$, $j = j_K$, $\tilde{T} = j_X$, and the second line is an extension sequence.
\end{props}
\begin{proof}
One easily checks that $p$ is well defined, the second line is exact and that the diagram commutes.

We begin by showing that $j$ is a complete isomorphism. Take $C_1 > 0$ such that $\|y\| \leq C_1 \|i_n y\|$ for every $n$ and for every $y \in M_n(Y)$. Let $k \in M_n(K)$ and $y \in M_n(Y)$ be arbitrary. Then
\begin{eqnarray*}
\|k\| & \leq & \|k - T_ny\| + \|T_ny\| \\
      & \leq & \|k - T_ny\| + C_1\|T\|_{cb}\|i_n y \| \\
      & \leq & (1 + C_1 \|T\|_{cb})\|(i_n y, k - T_ny)\|
\end{eqnarray*}

Since $y$ was arbitrary, $\|k\| \leq (1 + C_1 \|T\|_{cb}) \|j_n(k)\|$. Since $j$ is c.b., it is a complete isomorphism.

Let us show now that $p$ is c.b. Let $(x, k) + M_n(\Delta) \in M_n(\mathcal{PO})$. We have that
\begin{eqnarray*}
\|p_n((x, k) + M_n(\Delta))\| & = & \|q_n(x)\| \\
    & \leq & \|q\|_{cb}\inf\{\|x + i_n y\| : y \in M_n(Y)\} \\
    & \leq & \|q\|_{cb}\inf\{\|(x + i_n y, k - T_ny)\| : y \in M_n(Y)\} \\
    & = & \|q\|_{cb}\|(x, k) + M_n(\Delta)\|
\end{eqnarray*}

Therefore $\|p\|_{cb} \leq \|q\|_{cb}$. Let $\tilde{p}$ be the operator induced by $p$. We must show that it is a complete isomorphism. We already know that it is c.b.

Let $((x, k) + M_n(\Delta)) + M_n(j(K))$ be an element of $M_n(\mathcal{OP})/M_n(j(K))$. Then its norm is
\begin{eqnarray*}
&& \inf\{\|(x, k') + M_n(\Delta))\| : k' \in M_n(K)\} \\
  & = & \inf\{\|(x + i_n(y), k' - T_n(y))\| : k' \in M_n(K), y \in M_n(Y)\}
\end{eqnarray*}

Taking $k' = T_n(y)$, and using the fact that $q$ is a complete quotient map, there is a constant $C_2$ such that
\begin{eqnarray*}
\|((x, k) + M_n(\Delta)) + M_n(j(K))\| & \leq & \inf\{\|x + i_n(y)\| : y \in M_n(Y)\} \\
 & \leq & C_2\|q_n(x)\| \\
 & = & C_2 \|\tilde{p}_n(((x, k) + M_n(\Delta)) + M_n(j(K)))\|
\end{eqnarray*}

Therefore, $\tilde{p}$ is a complete isomorphism, and the second line is an extension sequence.
\end{proof}

We remark that if the first line induces a complete isometry between $Z$ and $X / i(Y)$, then the second line induces a complete isometry between $Z$ and $\mathcal{PO}/j(K)$. 

For the record, we state and prove:
\begin{props}\label{prop:2.18}
In the situation of \hyperref[prop:2.17]{Proposition 2.17}, if $T$ is a complete isomorphism onto its image, so is $\tilde{T}$.
\end{props}
\begin{proof}
We already know that $\tilde{T}$ is completely bounded. Let $x \in M_n(X)$. We have that $\tilde{T}_n(x) = (x, 0) + M_n(\Delta)$. Since $i$ is completely bounded and $T$ is a complete isomorphism, there is a constant $C > 0$ such that, for every $y \in M_n(Y)$,
\begin{equation*}
\|x\|  \leq  \|x + i_n y\| + C\|-T_n y\| \leq (1 + C)\|(x + i_n y, -T_n y)\|
\end{equation*}

By taking the infimum over all $y \in M_n(Y)$, we have that $\|x\| \leq (1 + C)\|\tilde{T}_n(x)\|$.
\end{proof}

\section{Complete twisted sums induced by complex interpolation}\label{sec:3}
\subsection{Complex interpolation and extension sequences}\label{sec:3.1} We recall Calder\'{o}n's complex method of interpolation. A general reference for the topic is \cite{Bergh01}.

Let $\overline{X} = (X_0, X_1)$ be a \textit{compatible couple} of Banach spaces, that is, $X_0$ and $X_1$ are continually and linearly embedded in a Hausdorff topological space $U$. The space $U$ may be replaced by the space
\[
\Sigma(\overline{X}) = \{x_0 + x_1 \in U : x_0 \in X_0, x_1 \in X_1\}
\]

The space $\Sigma(\overline{X})$ is equipped with the complete norm
\[
\|x\| = \inf\{\|x_0\|_{X_0} + \|x_1\|_{X_1} : x = x_0 + x_1, x_0 \in X_0, x_1 \in X_1\}
\]

Let $\mathcal{F} = \mathcal{F}(\overline{X})$ be the space of functions $f$ defined on $\mathbb{S} = \{z \in \mathbb{C} : 0 \leq Re(z) \leq 1\}$ with image in $\Sigma(\overline{X})$ such that:
\begin{enumerate}
\item[(F1)] $f$ is continuous and bounded on $\mathbb{S}$ and analytic in the interior of $\mathbb{S}$;
\item[(F2)] $f(it) \in X_0$, $f(1 + it) \in X_1$, for every $t \in \mathbb{R}$;
\item[(F3)] $t \mapsto \|f(it)\|_{X_0}$ and $t \mapsto \|f(1 + it)\|_{X_1}$ are continuous bounded functions from $\mathbb{R}$ to $\mathbb{R}$.
\end{enumerate}

For $f \in \mathcal{F}$, let
\[
\|f\| = \max\{\sup\limits_{t \in \mathbb{R}} \|f(it)\|_{X_0}, \sup\limits_{t \in \mathbb{R}} \|f(1 + it)\|_{X_1}\}
\]

With this norm, $\mathcal{F}$ is a Banach space, and for $\theta \in [0, 1]$, the evaluation $\delta_{\theta} : \mathcal{F} \rightarrow \Sigma(\overline{X})$ at $\theta$ is continuous.

Then the space $X_{\theta} = (X_0, X_1)_{\theta} = \mathcal{F}/\ker(\delta_{\theta})$ is a Banach space called the \textit{interpolation space of} $(X_0, X_1)$ \textit{at} $\theta$. That is
\[
X_{\theta} = \{f(\theta) : f \in \mathcal{F}\}
\]

\noindent normed by
\[
\|x\| = \inf\{\|f\| : f \in \mathcal{F}, f(\theta) = x\}
\]

Complex interpolation is also defined for operator spaces \cite{Pisier03}.
If $(X_0, X_1)$ is a compatible pair of Banach spaces and $X_0$ and $X_1$ are operator spaces, for each $n$ we may see $(M_n(X_0), M_n(X_1))$ as a compatible couple by the inclusions $M_n(X_0), M_n(X_1) \subset \Sigma(\overline{X})^n$, because of \eqref{eq:1.2}. Then the operator space structure of $X_{\theta}$ is defined by the isometric identifications
\[
M_n(X_{\theta}) = (M_n(X_0), M_n(X_1))_{\theta}
\]

Another way to see this is by considering on $\mathcal{F}$ the operator space structure
\[
M_n(\mathcal{F}(\overline{X})) = \mathcal{F}(M_n(X_0), M_n(X_1))
\]

This operator space structure is induced by the inclusion (see \cite{Blecher01})
\[
\mathcal{F} \subset L_{\infty}(\mathbb{R}, X_0) \oplus_{\infty} L_{\infty}(\mathbb{R}, X_1)
\]

Then $X_{\theta} = \mathcal{F}/\ker(\delta_{\theta})$ completely isometrically. This interpolation space satisfies properties analogous to the classic interpolation of Banach spaces. For example, we have the Riesz-Thorin Theorem: if $T$ is an operator completely bounded on $X_0$ and on $X_1$, then $T$ is completely bounded on $X_{\theta}$ (see \cite{Pisier03}, Proposition 2.1).

The following lemma is well known in the Banach space case, see, for example, \cite{Castillo01}.

\begin{lemma}\label{lem:3.1}
Let $(X_0, X_1)$ be a compatible couple of operator spaces. For $\theta \in (0, 1)$ the evaluation of the derivative at $\theta$, $\delta_{\theta}' : \ker(\delta_{\theta}) \rightarrow X_{\theta}$, is onto and completely bounded.
\end{lemma}
\begin{proof}
Let $\mathbb{D}= \{z \in \mathbb{C} : \left|z\right| \leq 1\}$ and let $\varphi : \mathbb{S} \rightarrow \mathbb{D}$ be a conformal equivalence with $\varphi(\theta) = 0$.

Let $f \in M_n(\mathcal{F})$ with $f(\theta) = 0$. Then $f = \varphi g$ for some $g \in M_n(\mathcal{F})$, $f'(\theta) = \varphi'(\theta) g(\theta)$, and
\begin{eqnarray*}
\|f'(\theta)\| & = & \left|\varphi'(\theta)\right| \|g(\theta)\| \\
   & \leq & \left|\varphi'(\theta)\right| \|g\| \\
   & = & \left|\varphi'(\theta)\right| \|\varphi g\| \\
   & = & \left|\varphi'(\theta)\right| \|f\|
\end{eqnarray*}

Thus $\delta_{\theta}'$ is completely bounded. To see that it is onto, given $x \in X_{\theta}$, take $f \in \mathcal{F}$ such that $f(\theta) = x$ and consider $g = \frac{\varphi f}{\varphi'(\theta)}$.
\end{proof}

So, in the context of operator spaces, we have a diagram
\[
\xymatrix{
0 \ar[r] & \ker(\delta_{\theta}) \ar[r]^{i}\ar[d]^{\delta_{\theta}'} & \mathcal{F} \ar[r]^{\delta_{\theta}} & X_{\theta} \ar[r] & 0 \\
& X_{\theta}}
\]

\noindent where the first line is an extension sequence. Using the pushout \hyperref[prop:2.17]{(Proposition 2.17)}, we get a commutative diagram
\[
\xymatrix{
0 \ar[r] & \ker(\delta_{\theta}) \ar[r]^{i}\ar[d]^{\delta_{\theta}'} & \mathcal{F} \ar[d]\ar[r]^{\delta_{\theta}} & X_{\theta} \ar@{=}[d]\ar[r] & 0 \\
0 \ar[r] & X_{\theta} \ar[r]^i & \mathcal{PO} \ar[r]^q & X_{\theta} \ar[r] & 0}
\]

\noindent where the lines are extension sequences. That is, we get a complete extension of $X_{\theta}$.

As in the Banach space case, we have a simple description of this extension:
\begin{props}\label{prop:3.2}
In the above situation, we have a complete equivalence of complete twisted sums
  \[
  \xymatrix{ 0 \ar[r] & X_{\theta} \ar[r]^i\ar@{=}[d] & \mathcal{PO}
    \ar[r]^q\ar[d]^{T} & X_{\theta}\ar[r] \ar@{=}[d] &0 \\ 0 \ar[r] & X_{\theta}
    \ar[r]^j & dX_{\theta} \ar[r]^p& X_{\theta}\ar[r] &0}
  \]
  
\noindent where
\[
dX_{\theta} = \mathcal{F}/(\ker(\delta_{\theta}) \cap \ker(\delta'_{\theta})) = \{(f(\theta), f'(\theta)) : f \in \mathcal{F}\}
\]

\noindent with the quotient norm, $j(y) = (0, y)$ and $p(x, y) = x$.
\end{props}
\begin{proof}
First, we prove that $dX_{\theta}$ is a complete extension of $X_{\theta}$. 

Let $\varphi : \mathbb{S} \rightarrow \mathbb{D}$ be a conformal equivalence with $\varphi(\theta) = 0$. Notice that $\ker \delta_{\theta} = \varphi \mathcal{F}$ and $\ker (\delta_{\theta}) \cap \ker (\delta'_{\theta}) = \varphi^2 \mathcal{F}$.

Let $y \in M_n(X_{\theta})$, and $f \in M_n(\mathcal{F})$ with $f(\theta) = y$. Let $f_1 = \frac{\varphi f}{\varphi'(\theta)}$. Since $f_1(\theta) = 0$ and $f_1'(\theta) = y$, we have:
\begin{eqnarray*}
\|j_n(y)\| & = & \frac{1}{\left|\varphi'(\theta)\right|}\|((\varphi f)(\theta), (\varphi f)'(\theta))\| \\
 & = & \frac{1}{\left|\varphi'(\theta)\right|} \inf\{\|\varphi f + \varphi^2 g\| : g \in M_n(\mathcal{F})\} \\
 & = & \frac{1}{\left|\varphi'(\theta)\right|} \inf\{\|f + \varphi g\| : g \in M_n(\mathcal{F})\} \\
 & = & \frac{1}{\left|\varphi'(\theta)\right|} \inf\{\|f + h\| : h \in \ker (\delta_{\theta})_n\} \\
 & = & \frac{1}{\left|\varphi'(\theta)\right|} \|y\|
\end{eqnarray*}

Therefore, $j$ is a complete isomorphism onto its image. Now, to see that $p$ is completely bounded, let $f \in M_n(\mathcal{F})$. Then:
\[
\|p_n(f(\theta), f'(\theta))\| = \|f(\theta)\| = \inf\{\|g\| : g(\theta) = f(\theta)\} \leq \|(f(\theta), f'(\theta))\|
\]

Finally, it is a complete quotient map, since, for all $f \in M_n(\mathcal{F})$, we have:
\begin{eqnarray*}
\|\tilde{p}_n((f(\theta), f'(\theta)) + j_n(X_{\theta}))\| & = & \|f(\theta)\| \\
    & = & \inf\{\|f + g\| : g \in \ker (\delta_{\theta})_n\} \\
    & = & \inf\{\|h\| : h(\theta) = f(\theta), h'(\theta) - f'(\theta) \in M_n(X_{\theta})\} \\
    & = & \|((f(\theta), f'(\theta)) + j_n(X_{\theta})\|
\end{eqnarray*}

The operator $T$ is defined by $T((f, x) + \Delta) = (f(\theta), x + f'(\theta))$. One easily checks that it is well-defined and makes the diagram commute. It remains to show that it is a complete isomorphism. By the $3-$lemma \hyperref[prop:2.3]{(Proposition 2.3)}, it is enough to prove that $T$ or $T^{-1}$ is completely bounded, but we show both, since we can get universal constants for the c.b. norms of $T$ and $T^{-1}$.

First, we show that $T$ is completely bounded. Let $(f, x) + M_n(\Delta) \in M_n(\mathcal{PO})$, and take $g \in M_n(\mathcal{F})$ such that $g(\theta) = 0$ and $g'(\theta) = x$. Then:
\[
T_n((f, x) + M_n(\Delta)) = (f(\theta), x + f'(\theta)) = ((f + g)(\theta), (f + g)'(\theta))
\]

Therefore
\[
\|T_n((f, x) + M_n(\Delta))\| = \inf\{\|f + g + \varphi h\| : h \in \ker (\delta_{\theta})_n\}
\]

Recall that
\[
\|(f, x) + M_n(\Delta)\| = \inf\{\|(f + l, x - l'(\theta)\| : l \in \ker (\delta_{\theta})_n\}
\]

Given $l \in \ker(\delta_{\theta})_n$, since $g - l \in \ker(\delta_{\theta})_n$, there is $h_1 \in M_n(\mathcal{F})$ such that $\varphi h_1 = g - l$, and $h_1(\theta) = \frac{x - l'(\theta)}{\varphi'(\theta)}$.

For $\epsilon > 0$, take $h_{l, \epsilon} \in \ker(\delta_{\theta})_n$ such that $\|h_1 + h_{l, \epsilon}\| \leq \|\frac{x - l'(\theta)}{\varphi'(\theta)}\| + \epsilon$. Then:
\begin{eqnarray*}
\|f + g + \varphi h_{l, \epsilon}\| & \leq & \|f + l\| + \|g - l + \varphi h_{l, \epsilon}\| \\
    & = & \|f + l\| + \|h_1 + h_{l, \epsilon}\| \\
    & \leq & (1 + \frac{1}{\left|\varphi'(\theta)\right|})\max\{\|f + l\|, \|x - l'(\theta)\|\} + \epsilon
\end{eqnarray*}

Since $l$ and $\epsilon>0$ were arbitrary, we get
\[
\|T_n((f, x) + M_n(\Delta))\| \leq (1 + \frac{1}{\left|\varphi'(\theta)\right|}) \|(f, x) + M_n(\Delta)\|
\]

\noindent and $T$ is completely bounded.

Finally, since
\[
\|T_n((f, x) + M_n(\Delta))\| = \inf\{\|h\| : h(\theta) = f(\theta), h'(\theta) = x + f'(\theta)\}
\]

\noindent given $h$ in the set above, we have $h - f \in \ker (\delta_{\theta})_n$ and $h'(\theta) - f'(\theta) = x$. If we let $g = h -f$, we have that $(g, -x) \in M_n(\Delta)$, and
\[
\|(f, x) + M_n(\Delta)\| \leq \|(f, x) + (g, -x)\| = \|(h, 0)\| = \|h\|
\]

Since $h$ was arbitrary, we get
\[
\|(f, x) + M_n(\Delta)\| \leq \|T_n((f, x) + M_n(\Delta))\|
\]

\noindent and $T$ is a complete isomorphism.
\end{proof}

Notice that by multiplying elements of $\mathcal{F}$ by $e^{\delta(z - \theta)^2}$, with $\delta < 0$, we may suppose that $\|f(it)\|_{X_0} \rightarrow 0$ and $\|f(1 +it)\|_{X_1} \rightarrow 0$ when $\left|t\right| \rightarrow 0$, without changing $f(\theta)$ and $f'(\theta)$, and the resulting norms on $X_{\theta}$ and $dX_{\theta}$ are the same. This allows us to use results from \cite{Bergh01}.

\subsection{Duality}\label{sec:3.2}
Let $(X_0, X_1)$ be a compatible couple of operator spaces, and let $\Delta(\overline{X}) = X_0 \cap X_1$, normed by
\[
\|x\| = \max\{\|x\|_0, \|x\|_1\}
\]

If $\Delta(\overline{X})$ is dense in both $X_0$ and $X_1$, and $X_0$ or $X_1$ is reflexive, then $(X_0^*, X_1^*)$ is also an interpolation couple and $(X_0^*, X_1^*)_{\theta} = (X_{\theta})^*$ completely isometrically (Theorem 2.2 of \cite{Pisier03} and \cite{Bergh02}). Because of this identification we can drop the parenthesis and write $X_{\theta}^*$. This interpolation scheme induces an extension sequence:
  \[
  \xymatrix{ 0 \ar[r] & X_{\theta}^* \ar[r] & d(X_{\theta}^*)
    \ar[r] & X_{\theta}^*\ar[r]  &0}
  \]

On the other hand, by dualizing the extension sequence induced by the interpolation scheme $(X_0, X_1)$, we get an extension sequence
  \[
  \xymatrix{ 0 \ar[r] & X_{\theta}^* \ar[r] & (dX_{\theta})^*
    \ar[r] & X_{\theta}^*\ar[r]  &0}
  \]
  
In the Banach space setting these are equivalent twisted sums of $X_{\theta}^*$ (see \cite{Cwikel01, Rochberg01})). We shall prove now that the same result holds in the operator space case (notice that we are treating the complete isometric identification between $(X_{\theta})^*$ and $(X_0^*, X_1^*)_{\theta}$ as an identity). We will use this result later to prove complete cosingularity of some extension sequences.

One crucial part of the proof in the Banach space case is the Schwarz-Pick lemma (\cite{Krantz1999handbook}, for example). We will need a noncommutative version of it.
\begin{lemma}(Noncommutative Schwarz-Pick Lemma)\label{lem:3.3}
Let $F : \mathbb{D}^o \rightarrow M_n$ be an analytic function such that $\|F(z)\| \leq 1$ for every $z \in \mathbb{D}^o$. Then
\[
\|F'(z)\| \leq \frac{1}{1 - \left|z\right|^2}
\]

\noindent for every $z \in \mathbb{D}^o$.
\end{lemma}
\begin{proof}
Given $x, y \in \ell_2^n$ of norm $1$, let $H_{x, y} : \mathbb{D}^o \rightarrow \mathbb{C}$ be given by $H_{x, y}(z) = <F(z)x, y>$. Then $H_{x, y}$ is analytic and $\left|H_{x, y}(z)\right| \leq 1$ for all $z \in \mathbb{D}^o$.

By the Schwarz-Pick lemma:
\[
\left|H'_{x, y}(z)\right| \leq \frac{1 - \left|H_{x, y}(0)\right|^2}{1 - \left|z\right|^2} \leq \frac{1}{1 - \left|z\right|^2}
\]

But $H'_{x, y}(z) = <F'(z)x, y>$, and since $x$ and $y$ were arbitrary of norm $1$, we have the result.
\end{proof}

\begin{theorem}\label{teo:3.4}
Suppose that $\Delta(\overline{X})$ is dense in $X_0$ and in $X_1$ and that at least one of the spaces $X_0$ or $X_1$ is reflexive. Then the extension sequences
  \[
  \xymatrix{ 0 \ar[r] & X_{\theta}^* \ar[r] & d(X_{\theta}^*)
    \ar[r] & X_{\theta}^*\ar[r]  &0}
  \]
and
  \[
  \xymatrix{ 0 \ar[r] & X_{\theta}^* \ar[r] & (dX_{\theta})^*
    \ar[r] & X_{\theta}^*\ar[r]  &0}
  \]
\noindent are completely (isomorphically) equivalent.
\end{theorem}
\begin{proof}
We must define a completely bounded map $T : d(X_{\theta}^*) \rightarrow (dX_{\theta})^*$ making the diagram commute. By Lemma 4.2.3 of \cite{Bergh01}, we have that the sets
\begin{eqnarray*}
A & = & \{(x^*, y^*) \in d(X_{\theta})^* : x^*, y^* \in \Delta(\overline{X^*})\} \\
B & = & \{(x, y) \in dX_{\theta} : x, y \in \Delta(\overline{X})\}
\end{eqnarray*}
\noindent are dense in $d(X_{\theta}^*)$ and $dX_{\theta}$, respectively. For $(x^*, y^*)$ and $(x, y)$ in these sets, define
\[
T(x^*, y^*)(x, y) = x^*y + y^*x
\]

This is well-defined because $\Delta(\overline{X^*}) = \Sigma(\overline{X})^*$ and $\Sigma(\overline{X^*}) = \Delta(\overline{X})^*$ (\cite{Bergh01}, Theorem 2.7.1). It is a classical result that $T$ defines an isomorphism between $d(X_{\theta}^*)$ and $(dX_{\theta})^*$ (see \cite{Rochberg01}). We must show that it is completely bounded.

Recall that $M_n((dX_{\theta})^*) = CB(dX_{\theta}, M_n)$. 
Suppose that $(x^*, y^*) \in M_n(A)$ and $(x, y) \in B$, and take $f \in M_n(\mathcal{F}(\overline{X^*}))$ and $g \in \mathcal{F}(\overline{X})$ such that $f(\theta) = x^*, f'(\theta) = y^*, g(\theta) = x, g'(\theta) = y$, and $\|f\| \leq \|(x^*, y^*)\| + \epsilon$, $\|g\| \leq \|(x, y)\| + \epsilon$. 

By a modification of Stafney's argument (Lemma 2.5, \cite{Stafney01}, use $r^2$ instead of $r$), we may suppose that the entries of $f$ and $g$ are continuous functions with values in $\Delta(\overline{X^*})$ and $\Delta(\overline{X})$, respectively (actually, we may suppose they are functions of the form of the ones of Lemma 4.2.3 of \cite{Bergh01}).

We define the function $H : \mathbb{S} \rightarrow M_n$ given by
\[
H(z) = (<f_{i, j}(z), g(z)>)
\]
which is continuous, bounded (since each entry is bounded), and analytic in the interior of $\mathbb{S}$, with derivative given by
\[
H'(z) = (<f'_{i, j}(z), g(z)> + <f_{i, j}(z), g'(z)>)
\]

We have the estimate
\begin{eqnarray*}
\|H(it)\| & = & \|(<f_{i, j}(it), g(it)>)\| \\
          & = & \|f(it)(g(it))\| \\
          & \leq & \|f(it)\|_{M_n(X_0^*) = CB(X_0, M_n)} \|g(it)\|_{X_0} \\
          & \leq & (\|(x^*, y^*)\| + \epsilon)(\|(x, y)\| + \epsilon)
\end{eqnarray*}

\noindent and the same for $\|H(1+it)\|$, $t \in \mathbb{R}$. Therefore, $\|H(z)\| \leq (\|(x^*, y^*)\| + \epsilon)(\|(x, y)\| + \epsilon) = K$, for all $z \in \mathbb{S}$.

Let $\varphi : \mathbb{D} \rightarrow \mathbb{S}$ be a conformal map and $\psi : \mathbb{D} \rightarrow M_n$ be given by $\psi = \frac{H \circ \varphi}{K}$. Then $\psi$ satisfies the assumptions of \hyperref[lem:3.3]{Lemma 3.3}, and since
\[
\psi'(z) = \frac{\varphi'(z)H'(\varphi(z))}{K}
\]

\noindent we have
\begin{equation*}
\|H'(\theta)\| \leq C (\|(x^*, y^*)\| + \epsilon)(\|(x, y)\| + \epsilon)
\end{equation*}

\noindent for some universal constant $C > 0$. Since $\epsilon$ was arbitrary,
\begin{equation}\label{eq:3.1}
\|H'(\theta)\| \leq C \|(x^*, y^*)\|\|(x, y)\|
\end{equation}

But $H'(\theta) = T_n(x^*, y^*)(x, y)$. The map $T(x^*, y^*)$ was only defined for $(x, y) \in B$. By \eqref{eq:3.1}, it is continuous, and therefore extends to all $dX_{\theta}$.

Again by \eqref{eq:3.1}, $T$ extends to an operator from $d(X_{\theta}^*)$ into $(dX_{\theta})^*$ which is completely bounded, and by the $3-$lemma \hyperref[prop:2.3]{(Proposition 2.3)}, it is a complete isomorphism.
\end{proof}

We pass now to some basic examples.

\subsection{$O\ell_p$ spaces}\label{sec:3.3}
The interpolation scale $(\ell_{\infty}, \ell_1)_{\theta} = \ell_{\frac{1}{\theta}}$ induces for each $1 < p < \infty$ the Kalton-Peck space $Z_p$, which appears as an extension of $\ell_p$:
\begin{equation}\label{eq:3.2}
\xymatrix{
0 \ar[r] & \ell_p \ar[r] & Z_p \ar[r] & \ell_p \ar[r] & 0}
\end{equation}

For each $1 < p < \infty$ this is a singular cosingular twisted sum \cite{Kalton01}.

In \cite{Pisier05}, Pisier gives $\ell_p$ a natural operator space structure in the following way: recall that $\ell_{\infty}$ is a $C^*-$algebra, and therefore has a natural operator space structure that coincides with $\min(\ell_{\infty})$. Its predual, $\ell_1$, has then a natural operator space structure induced by $(\min(\ell_{\infty}))^*$, which coincides with $\max(\ell_1)$ \cite{Pisier04}.

Then for $1 < p < \infty$ we have a natural operator space structure on $\ell_p$ induced by the interpolation scale $(\min(\ell_{\infty}), \max(\ell_1))_{\frac{1}{p}} = O\ell_p$.

We also get a natural operator space structure on $Z_p$, for $1 < p < \infty$, for which \eqref{eq:3.2} becomes an extension sequence.

Since at the Banach level \eqref{eq:4.1} is singular and cosingular, it is completely singular and completely cosingular.

\subsection{Twisting $\min(\ell_p)$ and $\max(\ell_p)$}\label{sec:3.4}
By using first the pullback and then the pushout on \eqref{eq:4.1} we get a completely singular and completely cosingular complete twisted sum of $\min(\ell_p)$ and $\max(\ell_p)$. Indeed, using the pullback we get a commutative diagram
\[
\xymatrix{
& 0 & 0 & 0\\
0 \ar[r] & 0 \ar@{=}[d]\ar[r]\ar[u] & \max(\ell_p) \ar[r]^{Id_{\ell_p}}\ar[u] & O\ell_p \ar[r]\ar[u] & 0\\
0 \ar[r] & 0 \ar[r] & \Xi \ar[r]\ar[u] & Z_p \ar[u]\ar[r] & 0 \\
0 \ar[r] & 0 \ar[r]\ar[u] & O\ell_p \ar[r]\ar[u] & O\ell_p \ar[u]\ar[r] & 0\\
& 0 \ar[u] & 0 \ar[u] & 0 \ar[u]
}
\]

The second column is complete exact, and since in the Banach space level it is equivalent to the third, it is completely singular and completely cosingular.

Now, using the pushout, we get
  \[
  \xymatrix{ 0 \ar[r] & O\ell_p \ar[r]\ar[d]^{Id_{\ell_p}} & \Xi
    \ar[r]\ar[d]^{T} & \max(\ell_p)\ar[r] \ar@{=}[d] &0 \\ 0 \ar[r] & \min(\ell_p)
    \ar[r] & \mathcal{PO} \ar[r]& \max(\ell_p)\ar[r] &0}
  \]
  
For the same reason as before, the second line is completely singular and completely cosingular, and $\mathcal{PO}$ as a Banach space is simply $Z_p$.

Notice that this construction is always possible when we have a complete twisted sum which is singular/cosingular as a twisted sum of Banach spaces.

\subsection{Noncommutative $L_p$ spaces}\label{sec:3.5}
In \cite{CabelloSanchez2016}, Cabello S\'anchez, Castillo, Goldstein and Su\'arez de la Fuente study twisted sums induced by the interpolation scale $(\mathcal{M}, L_1(\mathcal{M}, \tau))_{\theta} = L_p(\mathcal{M}, \tau)$, $\theta = \frac{1}{p}$, where $\mathcal{M}$ is a von Neumann algebra and $\tau$ is a trace. They also consider twisted sums induced by general $\sigma-$finite algebras. (Notice that they consider the spaces as Banach modules over $\mathcal{M}$)

Since these noncommutative spaces also have natural operator space structures defined in \cite{Pisier05}, the resulting twisted sums also have an operator space structure induced by interpolation.

\section{$o(\ell_p)$ spaces}\label{sec:4}
\subsection{The isomorphic Palais' problem for operator spaces}\label{sec:4.1} In this section we give examples of complete twisted sums that are trivial in the Banach space level, but are not completely trivial. They are even completely singular and completely cosingular.

One of the central objects of Operator Space Theory is the Operator Hilbert Space $OH$, the only operator space structure on $\ell_2$ for which the canonical isometry between $\ell_2$ and its conjugate dual is a complete isometry.

The operator space $OH$ may be obtained in various ways. For example, under certain conditions we have that complex interpolation between an operator space and its dual at $\theta = \frac{1}{2}$ gives us $OH$ \cite{Pisier03}.

One of the most natural ways to obtain $OH$ is by the identity 
\[
OH = (\min(\ell_2), \max(\ell_2))_{\frac{1}{2}}
\]

So one is led to consider the following quantization of a Banach space $X$, as defined in \cite{Pisier03}: 
\[
o(X) = (\min(X), \max(X))_{\frac{1}{2}}
\]

More generally, we will consider $o(\ell_p)(\theta) = (\min(\ell_p), \max(\ell_p))_{\theta}$, for $0 < \theta < 1$. We then have complete twisted sums
\begin{equation}\label{eq:4.1}
\xymatrix{
0 \ar[r] & o(\ell_p)(\theta) \ar[r] & do(\ell_p)(\theta) \ar[r] & o(\ell_p)(\theta) \ar[r] & 0}
\end{equation}

Clearly in the Banach space level these are trivial twisted sums.
\begin{theorem}\label{teo:4.1}
The extension sequence \eqref{eq:4.1} is completely singular and completely cosingular for every $1 < p < \infty$ and for every $0 < \theta < 1$.
\end{theorem}

This section is devoted to proving \hyperref[teo:4.1]{Theorem 4.1}.

For each $n \in \mathbb{N}$ let $x^n \in M_n(c_{00})$ be the matrix with first line $(e_1 \; ... \; e_n)$ and $0$ elsewhere.

\begin{lemma}\label{lem:4.2}
For each $n \in \mathbb{N}$ and $1 \leq p \leq \infty$, we have:
\[
\|x^n\|_{M_n(\min(\ell_p))} = \begin{cases} n^{\frac{1}{p} - \frac{1}{2}}, & \mbox{if } 1 \leq p \leq 2 \\ 
1, & \mbox{if } p \geq 2 \end{cases}
\]
\end{lemma}
\begin{proof}
The proof is a simple application of formula \eqref{eq:1.3}.
\end{proof}

We will also consider the following sequence $(A_n)$ of scalar matrices: let $A_1$ be simply $(1)$, and given $A_n$, the matrix $A_{n+1}$ is given by
\[A_{n+1} =  \begin{pmatrix}
1 & \block(2,2){A_n} \\
\vdots & \block(2,2){-A_n}  \\
  1 & &&
\end{pmatrix} \]

The following is an easy consequence of the definition.
\begin{lemma}\label{lem:4.3}
For each $n \in \mathbb{N}$, we have:
\begin{enumerate}
\item The scalar product of any two differents columns of $A_n$ is $0$.
\item $\|A_n\| = (2^{n-1})^{\frac{1}{2}}$
\end{enumerate}
\end{lemma}

\begin{lemma}\label{lem:4.4}
For each $n \in \mathbb{N}$ and $1 \leq p \leq \infty$, we have:
\[
\|x^n\|_{M_n(\max(\ell_p))} = \begin{cases} n^{\frac{1}{2}}, & \mbox{if } 1 \leq p \leq 2 \\ 
n^{\frac{1}{p}}, & \mbox{if } p \geq 2 \end{cases}
\]
\end{lemma}
\begin{proof}
Let $A \in M_{n, 2^{n-1}}$ be given by
\[\frac{1}{2^{n-1}}
\begin{pmatrix}
1 & \hdots & 1 \\
0 & \hdots & 0  \\
\vdots & \ddots & \vdots \\
0 & \hdots & 0
\end{pmatrix}
\]

Writing $A_n = (a_{i, j}^{(n)})$, let $D \in M_{2^{n-1}}(\ell_p) = (d_{i, j})$ be the diagonal matrix given by $d_{i, i} = \sum\limits_{l=1}^n a_{i, l}^{(n)} e_l$.

Then a simple calculation shows that $x^n = ADA_n$, and by \eqref{eq:1.5}, $\|x^n\| \leq \|A\|\|D\|\|A_n\| = n^{\frac{1}{p}}$.

Since $\|x^n\|_{M_n(R)} = n^{\frac{1}{2}}$, we have that $\|x^n\|_{M_n(\max(\ell_2))} = n^{\frac{1}{2}}$.

Now let $1 \leq p < 2$. If $\|x^n\|_{M_n(\max(\ell_p))} < n^{\frac{1}{2}}$, we have a decomposition $x^n = ADB$, with $A$ and $B$ scalar matrices, $D \in M_n(\ell_p)$ a diagonal matrix, and $\|A\|\|D\|_{M_n(\ell_p)}\|B\| < n^\frac{1}{2}$. But $\|D\|_{M_n(\ell_2)} \leq \|D\|_{M_n(\ell_p)}$, which would imply $\|x^n\|_{M_n(\max(\ell_2))} < n^{\frac{1}{2}}$, a contradiction. Therefore, the case $1 \leq p \leq 2$ is proved.

If $p \geq 2$, recalling that $M_n(\max(\ell_p)) = M_n(\min(\ell_q)^*) = CB(\min(\ell_q), M_n)$ isometrically, where $\frac{1}{p} + \frac{1}{q} = 1$, we have
\begin{eqnarray*}
\|x^n\|_{M_n(\max(\ell_p))} & \geq & \frac{\|(x^n)_n (x^n)\|}{\|x^n\|_{\min(\ell_q)}} \\
        & = & \frac{n^{\frac{1}{2}}}{n^{\frac{1}{q} - \frac{1}{2}}} \\
        & = & n^\frac{1}{p}
\end{eqnarray*}

Therefore, $\|x^n\|_{M_n(\max(\ell_p))} = n^{\frac{1}{p}}$.
\end{proof}

\begin{lemma}\label{lem:4.5}
For each $n \in \mathbb{N}$, $1 \leq p < \infty$ and $\theta \in (0, 1)$, we have:
\[
\|x^n\|_{M_n(o(\ell_p)(\theta))} = \begin{cases} n^{\frac{1}{2} - (1-\theta)(1 - \frac{1}{p})}, & \mbox{if } 1 \leq p \leq 2 \\ 
n^{\frac{\theta}{p}}, & \mbox{if } p \geq 2 \end{cases}
\]
\end{lemma}
\begin{proof}
Since $\|x^n\|_{\min(\ell_1)} = \|x^n\|_{\max(\ell_1)} = n^{\frac{1}{2}}$, we have that $\|x^n\|_{M_n(o(\ell_1)(\theta))} = n^{\frac{1}{2}}$ for every $0 < \theta < 1$.

Let $1 < p < \infty$ and $0 < \theta < 1$ be fixed, and let
\[
\lambda_n = \begin{cases} n^{1 - \frac{1}{p}}, & \mbox{if } 1 < p \leq 2 \\ 
n^{\frac{1}{p}}, & \mbox{if } p \geq 2 \end{cases}
\]

Consider the interpolating function $f_n(z) = e^{(z-\theta)\log (\lambda_n^{-1})} x^n$. By lemmas \hyperref[lem:4.2]{4.2} and \hyperref[lem:4.4]{4.4}
\[
\|f_n\| = \begin{cases} n^{\frac{1}{2} - (1-\theta)(1 - \frac{1}{p})}, & \mbox{if } 1 \leq p \leq 2 \\ 
n^{\frac{\theta}{p}}, & \mbox{if } p \geq 2 \end{cases}
\]

Since $f_n(\theta) = x^n$, we obtain
\[
\|x^n\|_{M_n(o(\ell_p)(\theta))} \leq \begin{cases} n^{\frac{1}{2} - (1-\theta)(1 - \frac{1}{p})}, & \mbox{if } 1 < p \leq 2 \\ 
n^{\frac{\theta}{p}}, & \mbox{if } p \geq 2 \end{cases}
\]

Also, since $o(\ell_q)(1-\theta)^* = (\min (\ell_q), \max(\ell_q))_{1-\theta}^* = (\max(\ell_p), \min(\ell_p))_{1-\theta} = o\ell_p(\theta)$, for $\frac{1}{p} + \frac{1}{q} =1$, we have $M_n(o(\ell_p)(\theta)) = M_n(o(\ell_q)(1-\theta)^*) = CB(o(\ell_q)(1-\theta), M_n)$, and we obtain:\\

\textit{$1 < p < 2$}
\begin{eqnarray*}
\|x^n\|_{M_n(o(\ell_p)(\theta))} & \geq & \frac{\|(x^n)_n (x^n)\|}{\|x^n\|_{M_n(o(\ell_q)(1-\theta))}} \\
        & \geq & \frac{n^{\frac{1}{2}}}{n^{\frac{1-\theta}{q}}} \\
        & = & n^{\frac{1}{2} - (1-\theta)(1 - \frac{1}{p})}
\end{eqnarray*}

\textit{$p \geq 2$}
\begin{eqnarray*}
\|x^n\|_{M_n(o(\ell_p)(\theta))} & \geq & \frac{\|(x^n)_n (x^n)\|}{\|x^n\|_{M_n(o(\ell_q)(1-\theta))}} \\
        & \geq & \frac{n^{\frac{1}{2}}}{n^{\frac{1}{2} - \theta(1-\frac{1}{q})}} \\
        & = & n^{\frac{\theta}{p}}
\end{eqnarray*}
\end{proof}

\begin{lemma}\label{lem:4.6}
Let $W$ be a complemented subspace of $\ell_p$ ($1 \leq p < \infty$). Then $W$ with the operator space structure inherited from $o(\ell_p)(\theta)$ is completely isomorphic to $o(\ell_p)(\theta)$. In particular, every closed infinite dimensional subspace of $o(\ell_p)(\theta)$ contains a further subspace completely isomorphic to $o(\ell_p)(\theta)$.
\end{lemma}
\begin{proof}
Since $W$ is complemented, there exists $T : W \rightarrow \ell_p$ an onto isomorphism. Let $i : W \rightarrow \ell_p$ the inclusion, and $P : \ell_p \rightarrow W$ a bounded projection onto $W$.

Since $T \circ P : \ell_p \rightarrow \ell_p$ is bounded, by \eqref{eq:1.6} and the Riesz-Thorin Theorem for interpolation of operator spaces, $T \circ P : o(\ell_p)(\theta) \rightarrow o(\ell_p)(\theta)$ is completely bounded.

Since $i$ is a complete isometry, we get that $T = T\circ P \circ i$ is completely bounded.

For every $n \in \mathbb{N}$ and every $x \in M_n(\ell_p)$,
\[
\|T_n^{-1}(x)\|_{M_n(W)} = \|T_n^{-1}(x)\|_{M_n(o(\ell_p)(\theta))} = \|i_n \circ T_n^{-1}(x)\|_{M_n(o(\ell_p)(\theta))}
\]

Since $i \circ T^{-1} : \ell_p \rightarrow \ell_p$ is bounded, it is completely bounded from $o(\ell_p)(\theta)$ into $o(\ell_p)(\theta)$. So $T^{-1} : o(\ell_p)(\theta) \rightarrow W$ is completely bounded, and $T$ is a complete isomorphism.
\end{proof}

Let $\varphi : \mathbb{S} \rightarrow \mathbb{D}$ be a conformal map such that $\varphi(\theta) = 0$, and let $\beta = \frac{1}{\left|\varphi'(\theta)\right|}$. We shall use the following lemma, which is an easy adaptation of Lemma 2.9 of \cite{Rochberg01}.

\begin{lemma}\label{lem:4.7}
Let $(X_0, X_1)$ be a compatible pair of operator spaces, and suppose that $X_0 = X_1$ as Banach spaces. Let $n \in \mathbb{N}$ and $x, y_0 \in M_n(X_{\theta})$ $(x \neq 0)$ be such that there is $F \in \mathcal{F}$ with $F(\theta) = x$, $F'(\theta) = y_0$ and $\|F\|_{M_n(\mathcal{F})} = \|x\|_{M_n(X_{\theta})}$. Then for all $y \in M_n(X_{\theta})$ we have:
\[
\frac{1}{4}(\|x\|_{M_n(X_{\theta})} + \beta\|y - y_0\|_{M_n(X_{\theta})}) \leq \|(x, y)\|_{M_n(dX_{\theta})} \leq \|x\|_{M_n(X_{\theta})} + \beta\|y - y_0\|_{M_n(X_{\theta})}
\]
\end{lemma}
\begin{proof}
Recall from the proof of \hyperref[prop:3.2]{Proposition 3.2} that $y \in M_n(dX_{\theta})$, $\|(0, y)\|_{M_n(dX_{\theta})} = \beta \|y\|_{M_n(X_{\theta})}$. The proof then follows from the inequality
\[
\frac{\left|\|x\|_{M_n(X_{\theta})} - \beta\|y - y_0\|_{M_n(X_{\theta})}\right|}{\|x\|_{M_n(X_{\theta})} + \beta \|y - y_0\|_{M_n(X_{\theta})}} \leq \frac{\|(x, y)\|_{M_n(dX_{\theta})}}{\|x\|_{M_n(X_{\theta})} + \beta \|y - y_0\|_{M_n(X_{\theta})}}
\]
\noindent for the case $\beta\|y - y_0\|_{M_n(X_{\theta})} > 2\|x\|_{M_n(X_{\theta})}$ and from
\[
\frac{\|x\|_{M_n(X_{\theta})}}{\|x\|_{M_n(X_{\theta})} + \beta \|y - y_0\|_{M_n(X_{\theta})}} \leq \frac{\|(x, y)\|_{M_n(dX_{\theta})}}{\|x\|_{M_n(X_{\theta})} + \beta \|y - y_0\|_{M_n(X_{\theta})}}
\]
\noindent for the case $\beta \|y - y_0\|_{M_n(X_{\theta})} \leq 2 \|x\|_{M_n(X_{\theta})}$.
\end{proof}

\begin{props}\label{prop:4.8}
Let $1 < p < \infty$, $0 < \theta < 1$, and let $X = \{(x, 0) : x \in o(\ell_p)(\theta)\} \subset do(\ell_p)(\theta)$. Then $X$ is a closed subspace of $do(\ell_p)(\theta)$ which does not contain a completely isomorphic copy of $o(\ell_p)(\theta)$.
\end{props}
\begin{proof}
Let $W$ be a closed infinite dimensional subspace of $X$ and suppose that it is completely isomorphic to $o(\ell_p)(\theta)$. Clearly, $W$ is of the form
\[
W = \{(x, 0) \in do(\ell_p)(\theta) : x \in W'\}
\]

\noindent where $W'$ is a closed infinite dimensional subspace of $\ell_p$ (it is closed since $W'$ is isometric to $W$, which is complete). Take $W''$ an infinite dimensional closed subspace of $W'$ complemented in $\ell_p$ by $P : \ell_p \rightarrow W''$. By restricting $P$ to $W'$, we induce a projection from $W$ onto
\[
\{(x, 0) \in do(\ell_p)(\theta) : x \in W''\}
\]
which is bounded. In particular, since $W$ is completely isomorphic to $o(\ell_p)(\theta)$, we have by \hyperref[lem:4.6]{Lemma 4.6} that the subspace above is completely isomorphic to $o(\ell_p)(\theta)$. So we may suppose that $W'$ is complemented in $\ell_p$, and therefore completely isomorphic to $o(\ell_p)(\theta)$ with its operator space structure induced by $o(\ell_p)(\theta)$.

Since $W$ is completely isomorphic to $o(\ell_p)(\theta)$, there is a complete isomorphism $T : o(\ell_p)(\theta) \rightarrow W$, which must be of the form $T(x) = (T'(x), 0)$, for $T' : \ell_p \rightarrow W'$ an onto isomorphism.

By the proof of \hyperref[lem:4.6]{Lemma 4.6}, $T'$ is a complete isomorphism between $o(\ell_p)(\theta)$ and $W'$ with its operator space structure inherited from $o(\ell_p)(\theta)$.

Let $f \in M_n(\mathcal{F}(\min(\ell_p), \max(\ell_p)))$ be such that $f(\theta) = T'_n(x^n)$, $f'(\theta) = 0$. Notice that if $P : \ell_p \rightarrow W'$ is a bounded projection, $\|P_nf\| \leq \|P\|\|f\|$, by \eqref{eq:1.6}.

Let $g = (T'_n)^{-1} \circ P_n f$. Then $g \in M_n(\mathcal{F}(\min(\ell_p), \max(\ell_p)))$, $g(\theta) = x^n$, $g'(\theta) = 0$, and
\[
\|g\| \leq \|(T')^{-1}\|_{cb}\|P\|\|f\|
\]

Since $f$ was arbitrary and $T$ is completely bounded, we obtain constants $C_1$ and $C_2$ such that for every $n \in \mathbb{N}$:
\begin{equation}\label{eq:4.2}
\|(x^n, 0)\|_{M_n(do(\ell_p)(\theta))} \leq C_1 \|(T'_n x^n, 0)\|_{M_n(do(\ell_p)(\theta))} \leq C_2 \|x^n\|_{M_n(o(\ell_p)(\theta))}
\end{equation}

By \hyperref[lem:4.7]{Lemma 4.7}, if we let $\lambda_n$ be as in the proof of \hyperref[lem:4.5]{Lemma 4.5}, i.e.,
\[
\lambda_n = \begin{cases} n^{1 - \frac{1}{p}}, & \mbox{if } 1 < p \leq 2 \\ 
n^{\frac{1}{p}}, & \mbox{if } p \geq 2 \end{cases}
\]

\noindent we have
\[
\|(x^n, 0)\|_{M_n(do(\ell_p)(\theta))} \geq \frac{1}{4}(1 + \beta \log \lambda_n) \|x^n\|_{M_n(o(\ell_p)(\theta))}
\]

Since $\lambda_n$ is not bounded in $n$, we get a contradiction with \eqref{eq:4.2}.
\end{proof}

Notice that for $p = 1$ we would have $\lambda_n = 1$ for every $n$, and the proof does not work.

\begin{lemma}\label{lem:4.9}
The space $o\ell_p(\theta)$ is completely isomorphic to $o\ell_p(\theta) \oplus o\ell_p(\theta)$, $p \in [1, \infty]$, $\theta \in [0, 1]$.
\end{lemma}
\begin{proof}
Using \eqref{eq:1.6} and the Riesz-Thorin Theorem, one shows that $T : \ell_p \rightarrow \ell_p \oplus \ell_p$ given by $T(e_{2n}) = (e_n, 0)$, $T(e_{2n+1}) = (0, e_n)$ is a complete isomorphism between $o\ell_p(\theta)$ and $o\ell_p(\theta) \oplus o\ell_p(\theta)$.
\end{proof}

We are now ready to prove \hyperref[teo:4.1]{Theorem 4.1}:
\begin{proof}
We have the complete twisted sum
\begin{equation*}
\xymatrix{
0 \ar[r] & o(\ell_p)(\theta) \ar[r] & do(\ell_p)(\theta) \ar[r] & o(\ell_p)(\theta) \ar[r] & 0}
\end{equation*}

Let $W$ be a closed infinite dimensional subspace of $o(\ell_p)(\theta)$ such that the induced complete twisted sum
\begin{equation*}
\xymatrix{
0 \ar[r] & o(\ell_p)(\theta) \ar[r] & Z \ar[r] & W \ar[r] & 0}
\end{equation*}

\noindent is completely trivial. By taking a further subspace, we may suppose $W$ complemented in $\ell_p$, and therefore completely isomorphic to $o(\ell_p)(\theta)$. In particular, $Z$ is completely isomorphic to $o(\ell_p)(\theta) \oplus o(\ell_p)(\theta)$, which is completely isomorphic to $o(\ell_p)(\theta)$, by \hyperref[lem:4.9]{Lemma 4.9}.

But $Z = \{(x, y) \in do(\ell_p)(\theta) : x \in W\}$, which by \hyperref[prop:4.8]{Proposition 4.8} has an infinite dimensional closed subspace that has no completely isomorphic copy of $o\ell_p(\theta)$. However, this cannot happen, by the last part of \hyperref[lem:4.6]{Lemma 4.6}. Therefore, the extension sequences \eqref{eq:4.1} are completely singular.

Finally, by \hyperref[teo:3.4]{Theorem 3.4}, for $\frac{1}{p} + \frac{1}{q} = 1$ we have a complete equivalence
  \[
  \xymatrix{ 0 \ar[r] & o\ell_q(1-\theta) \ar[r]\ar@{=}[d] & do\ell_q(1-\theta)
    \ar[r]\ar[d] & o\ell_q(1-\theta)\ar[r] \ar@{=}[d] &0 \\ 0 \ar[r] & (o\ell_p(\theta))^*
    \ar[r] & (do\ell_p(\theta))^* \ar[r]& (o\ell_p(\theta))^*\ar[r] &0}
  \]
  
Since the quotient map of the first line is c.s.s., by \hyperref[prop:2.14]{Proposition 2.14} so is the quotient map of the second line. By \hyperref[prop:2.15]{Proposition 2.15}, item (2), the inclusion map from $o\ell_p(\theta)$ into $do\ell_p(\theta)$ is c.s.c., and the proof is complete.
\end{proof}

In particular, we have solved the isomorphic version of Palais' problem for operator spaces:
\begin{theorem}\label{teo:4.10}
There is an operator space $X$ isomorphic to $\ell_2$ which is not completely isomorphic to $OH$ and has a completely isometric copy of $OH$ with respective quotient also completely isometric to $OH$.
\end{theorem}

Following the classical setting, we say that a property $P$ of operator spaces is a \textit{complete 3-space property} ($C3SP$ for short) if whenever we have an extension sequence
  \[
  \xymatrix{ 0 \ar[r] & Y \ar[r] & X
    \ar[r] & Z\ar[r]  &0}
  \]
  
\noindent where $Y$ and $Z$ have $P$, then $X$ has $P$.

Let $Y$ be a Banach space. A Banach space $X$ is said to be $Y-$\textit{saturated} if every closed infinite dimensional subspace of $X$ has an isomorphic copy of $Y$. Being $\ell_p-$saturated is a $3SP$ (\cite{Castillo02}, Theorem 3.2.d). We may analogously define the property of being $Y-$\textit{completely saturated}, where $Y$ is an operator space. Notice that \hyperref[lem:4.6]{Lemma 4.6} essentially says that
\begin{props}\label{prop:4.11}
For each $p \in [1, \infty)$, and $\theta \in (0,1)$, the space $o\ell_p(\theta)$ is $o\ell_p(\theta)-$completely saturated.
\end{props}

However, as a consequence of \hyperref[prop:4.8]{Proposition 4.8}, we get
\begin{props}\label{prop:4.12}
For each $p \in (1, \infty)$, and $\theta \in (0,1)$, being $o\ell_p(\theta)-$completely saturated is not a C3SP. In particular, being $OH-$completely saturated is not a C3SP.
\end{props}

We do not know if \hyperref[prop:4.12]{Proposition 4.12} is true for $p = 1$.

One may wonder why the proof that being $\ell_p-$saturated does not pass to the quantum case. For the proof presented in \cite{Castillo02}, the breaking point is the following: let $N$ and $M$ be closed subspaces of a Banach space $X$. If $N + M$ is closed, then we have the isomorphic identification:
\[
\frac{N + M}{N} \cong \frac{M}{M \cap N}
\]

If we let $X$ be as in \hyperref[prop:4.8]{Proposition 4.8} and $Y = \{(0, y) \in do\ell_p(\theta) : y \in o\ell_p(\theta)\}$, then $(X+Y)/Y = o\ell_p(\theta)$ completely isomorphically, but \hyperref[prop:4.8]{Proposition 4.8} tells us that $X = X/(X\cap Y)$ is not completely isomorphic to $o\ell_p(\theta)$. That is, the second isomorphism theorem does not need to hold for operator spaces, even when it holds in the Banach setting (which is probably already known).

Recall that the space $\ell_p$ has a natural operator space structure, $O\ell_p$ \hyperref[sec:3.2]{(Section 3.2)}.

\begin{question}\label{qu:4.13}
Are the spaces $O\ell_p$, with their natural operator space structure, $O\ell_p-$completely saturated, $1 \leq p < \infty$, $p \neq 2$?
\end{question}

\begin{question}\label{qu:4.14}
Is being $O\ell_p-$completely saturated a C3SP?
\end{question}

\subsection{The isometric Palais' problem for operator spaces}\label{sec:4.2}

Having solved the isomorphic version of Palais' problem for operator spaces, it is easy to obtain a solution to the isometric version. We thank Gilles Pisier for pointing out how one can obtain a complete extension of $OH$ which is isometric to a Hilbert space.

Let $z_0 \in \mathbb{S}^o$ and $\mu_{z_0}$ be the harmonic measure on $\partial\mathbb{S}$ with respect to $z_0$. Let $\mathbb{S}_j = \{z \in \mathbb{C} : Re(z) = j\}$ and $\mu_{z_0}^j$ be the probability measure defined on $\mathbb{S}_j$ by $\mu_{z_0}$, $j = 0, 1$.

Given a compatible couple $\overline{X} = (X_0, X_1)$ of Banach spaces, $0 < \theta < 1$ and $p \in [1, \infty)$, let $\mathcal{F}_p = \mathcal{F}_p(\overline{X})$ be the space of functions $f : \mathbb{S} \rightarrow \Sigma(\overline{X})$ such that:
\begin{itemize}
	\item[P1] $f$ is analytic on $\mathbb{S}^o$;
    \item[P2] $f|_{\mathbb{S}_j} \in L_p(S_j, \mu_{\theta}^j, X_j)$, $j = 0, 1$;
    \item[P3] For every $z_0 \in \mathbb{S}^o$, we have
    	\[
        f(z_0) = \int\limits_{\partial\mathbb{S}} f(z) dP_{z_0}(z)
        \]
    \item[P4] We have
    	\[
        \|f\|_p^p = (1 - \theta) \int\limits_{\mathbb{S}_0} \|f(z)\|_{X_0}^p d\mu_{\theta}^0(z) + \theta \int\limits_{\mathbb{S}_1} \|f(z)\|_{X_1}^p d\mu_{\theta}^1(z) < \infty
        \]
\end{itemize}

Then $\mathcal{F}_p$ is a Banach space and if we still denote by $\delta_{\theta}$ the evaluation at $\theta$, then we have isometrically $X_{\theta} = \mathcal{F}_p / \ker(\delta_{\theta})$ (this is a consequence of Theorem 8.24 of \cite{Pisier2016}).

Let $E$ and $F$ be operator spaces. The inclusion $E \oplus_1 F \subset (E^* \oplus_{\infty} F^*)^*$ induces an operator space structure on the space $E \oplus_1 F$. Using interpolation, the operator space structure on the space $E \oplus_p F$ for $1 < p < \infty$ may be defined as
\[
E \oplus_p F = (E \oplus_{\infty} F, E \oplus_1 F)_{\eta}
\]
\noindent where $\eta = \frac{1}{p}$ (see \cite{Pisier05}).

If $X_0$ and $X_1$ are operator spaces, then $\mathcal{F}_p$ has a natural operator space structure induced by the natural inclusion
\[
\mathcal{F}_p \subset (1-\theta)L_p(\mathbb{S}_0, \mu_{\theta}^0, X_0) \oplus_p \theta L_p(\mathbb{S}_1, \mu_{\theta}^1, X_1)
\]

It was communicated to the author by Gilles Pisier that he and Yanqi Qiu were able to prove that the identity $X_{\theta} = \mathcal{F}_p/\ker(\delta_{\theta})$ is actually a complete isometry. We present their proof of this fact here with their authorization.

\begin{theorem}[G. Pisier, Y. Qiu]\label{teo:4.15}
Let $\overline{X} = (X_0, X_1)$ be a compatible couple of operator spaces and let $0 < \theta < 1$. Then we have completely isometrically $X_{\theta} = \mathcal{F}_p/\ker(\delta_{\theta})$.
\end{theorem}
\begin{proof}
Let $S_p^n$ be the $p-$Schatten class of dimension $n$. In \cite{Pisier05}, given an operator space $X$, it is defined the operator space $S_p^n[X]$. By Lemma 1.7 of \cite{Pisier05}, an operator $u : X \rightarrow Y$ is a complete isometry if and only if $Id_{S_p^n} \otimes u : S_p^n[X] \rightarrow S_p^n[Y]$ is an isometry for every $n$.

By Corollary 1.4 of \cite{Pisier05} we have isometrically:
\[
(S_p^n[X_0], S_p^n[X_1])_{\theta} = S_p^n[X_{\theta}]
\]

Therefore, for $x \in S_p^n[X_{\theta}]$,
\[
\|x\|_{S_p^n[X_{\theta}]} = \inf\{((1-\theta)\|f|_{\mathbb{S}_0}\|^p_{L_p(\mathbb{S}_0, \mu_{\theta}^0, S_p^n[X_0])} + \theta \|f|_{\mathbb{S}_1}\|^p_{L_p(\mathbb{S}_1, \mu_{\theta}^1, S_p^n[X_1])})^{\frac{1}{p}}\}
\]

\noindent where the infimum is over all $f \in \mathcal{F}_p(S_p^n[X_0], S_p^n[X_1])$ such that $f(\theta) = x$.

By Proposition 2.1 (ii) of \cite{Pisier05},
\[
\|x\|_{S_p^n[X_{\theta}]} = \inf\{((1-\theta)\|f|_{\mathbb{S}_0}\|^p_{S_p^n[L_p(\mathbb{S}_0, \mu_{\theta}^0, X_0)]} + \theta \|f|_{\mathbb{S}_1}\|^p_{S_p^n[L_p(\mathbb{S}_1, \mu_{\theta}^1, X_1)]})^{\frac{1}{p}}\}
\]

Finally, by (2.10) of \cite{Pisier05},
\[
\|x\|_{S_p^n[X_{\theta}]} = \inf\{\|f\|_{S_p^n[\mathcal{F}_p]} : f(\theta) = x\}
\]

Therefore, $X_{\theta} = \mathcal{F}_p/\ker(\delta_{\theta})$ completely isometrically.
\end{proof}

\begin{lemma}\label{lem:4.16}
Let $\varphi : \mathbb{S} \rightarrow \mathbb{D}$ be a conformal equivalence with $\varphi(\theta) = 0$. Then the operator $T_{\varphi} : \mathcal{F}_p \rightarrow \ker(\delta_{\theta}) \subset \mathcal{F}_p$ given by $T_{\varphi}(f) = \varphi f$ is a complete isometry.
\end{lemma}
\begin{proof}
Let $f \in \mathcal{F}_p$. To see that $T_{\varphi}(f) \in \ker(\delta_{\theta})$, we check P3, the other properties being clear.

For $z \in \mathbb{D}^o$ let $dP_{z}$ be the normalized harmonic measure on $\partial\mathbb{D}$ with respect to $z$. Let $w \in \mathbb{S}^o$. Then for every $g : \partial\mathbb{D} \rightarrow \mathbb{C}$ integrable with respect to $dP_{\varphi(w)}$ we have
\begin{equation}\label{eq:4.3}
\int\limits_{\partial\mathbb{S}} g(\varphi(z)) d\mu_w(z) = \int\limits_{\partial\mathbb{D}} g(z) dP_{\varphi(w)}(z)
\end{equation}

Given $x^* \in \Sigma(\overline{X})^*$, let $g : \mathbb{S} \rightarrow \mathbb{C}$ be given by $g(z) = x^*(z f(\varphi^{-1}(z)))$. Since $f$ satisfies P3 and $z$ is a bounded function on $\mathbb{D}$, we get by Theorem 3.1 of \cite{Duren2000} that $g$ is a function in the Hardy space $H^1$, and by \eqref{eq:4.3} and the fact that $x^*$ was arbitrary, we have that $\varphi f$ satisfies P3. Also, $\|\varphi f\|_{p} = \|f\|_p$.

Now, given $f \in \ker(\delta_{\theta})$ and $x^* \in \Sigma(\overline{X})^*$, we have that $g = x^*(f \circ \varphi^{-1})$ is in $H^1$, and therefore is in the Nevanlinna class $N^+$. This implies that $\frac{g}{z} \in N^+$ (\cite{Correa01}, Lemma 1.1), and by Theorem 2.11 of \cite{Duren2000}, we have that $\frac{g}{z} \in H^1$. Using Theorem 3.1 of \cite{Duren2000}, since $x^*$ was arbitrary, we have by \eqref{eq:4.3} that $\frac{f}{\varphi}$ satisfies P3. Also, $\|\frac{f}{\varphi}\|_p = \|f\|_p$.

Using Lemma 1.7, Proposition 2.1 (ii) and (2.10) of \cite{Pisier05}, we get that $T_{\varphi}$ is a complete isometry.
\end{proof}

Let $\delta_{\theta}'$ also denote the evaluation of the derivative at $\theta$ in $\mathcal{F}_p$. Let $dX_{\theta}^p = \mathcal{F}_p/(\ker(\delta_{\theta})\cap\ker(\delta_{\theta}^p)) = \{(f(\theta), f'(\theta)) : f \in \mathcal{F}_p\}$ with the quotient norm, and let $i^p : X_{\theta} \rightarrow dX_{\theta}^p$ and $q^p : dX_{\theta}^p \rightarrow X_{\theta}$ be given by $i^p(y) = (0, y)$ and $q^p(x, y) = x$, respectively.

\begin{props}\label{prop:4.17}
We have a complete extension of $X_{\theta}$
\[
  \xymatrix{ 0 \ar[r] & X_{\theta} \ar[r]^{i^p} & dX_{\theta}^p
    \ar[r]^{q^p} & X_{\theta} \ar[r]  &0}
  \]
  
\noindent and a complete equivalence of complete extensions
\[
  \xymatrix{ 0 \ar[r] & X_{\theta} \ar[r]\ar@{=}[d] & dX_{\theta}
    \ar[r]\ar[d]^{T} & X_{\theta}\ar[r] \ar@{=}[d] &0 \\ 0 \ar[r] & X_{\theta}
    \ar[r] & dX_{\theta}^p \ar[r]& X_{\theta}\ar[r] &0}
  \]

\noindent where $T(x, y) = (x, y)$.
\end{props}
\begin{proof}
Using \hyperref[lem:4.16]{Lemma 4.16} we may prove that $dX_{\theta}^p$ is a complete extension of $X_{\theta}$ in the same way that it was proved that $dX_{\theta}$ is a complete extension of $X_{\theta}$ \hyperref[prop:3.2]{(Proposition 3.2)}.

Since the inclusion $\mathcal{F} \subset \mathcal{F}_p$ is completely bounded, we get that $T$ is completely contractive. By the $3-$lemma, the two complete extensions are completely equivalent.
\end{proof}

Since $\min(\ell_2)$ and $\max(\ell_2)$ are Hilbertian, $\mathcal{F}_2$ is Hilbertian, and so is $do(\ell_2)(\frac{1}{2})^2$. By the previous proposition, this is a complete extension of $OH$ completely equivalent to $do(\ell_2)(\frac{1}{2})$, and therefore the two complete extensions have the same properties regarding complete singularity and complete cosingularity. This way, the isometric version of Palais' problem for operator spaces is solved.

\section{Complete extensions of $OH$}\label{sec:5}
We may also obtain the operator Hilbert Space $OH$ by the interpolation scale $(R, C)$. We have $OH = (R, C)_{\frac{1}{2}}$. Let $R(\theta) = (R, C)_{\theta}$ for $0 < \theta < 1$.

In the previous section we showed that the interpolation scale $(\min(\ell_2), \max(\ell_2))$ induces a completely singular/cosingular extension of $(\min(\ell_2), \max(\ell_2))_{\theta}$, for $\theta \in (0, 1)$. 

Notice that, if we let $x^n \in M_n(c_{00})$ be as in the previous section, we have:
\begin{eqnarray*}
\|x^n\|_{M_n(R)} & = & \|x^n\|_{M_n(\max(\ell_2))} = n^{\frac{1}{2}} \\
\|x^n\|_{M_n(C)} & = & \|x^n\|_{M_n(\min(\ell_2))} = 1
\end{eqnarray*}

Since $R = C^*$ and $C = R^*$ completely isometrically, we may mimic the proofs of the previous section and show that $f_n(z) = e^{(z-\theta)\log \sqrt{n}} x^n$ is a function in $M_n(\mathcal{F}(R, C))$ such that $f_n(\theta) = x^n$ and $\|f_n\| = \|x^n\|_{M_n(R(\theta))}$.

Also, the spaces $R(\theta)$ are homogeneous Hilbertian operator spaces, which implies that each closed infinite dimensional subspace is completely isometric to the whole space. This means that the proof that $do(\ell_p)(\theta)$ is completely singular and completely cosingular also works in this context, and we obtain:
\begin{theorem}\label{teo:5.1}
The interpolation scale $(R, C)$ induces a completely singular (cosingular) complete extension of $R(\theta)$, for $\theta \in (0, 1)$.
\end{theorem}

Let $dOH$ be the complete extension of $OH$ obtained by the interpolation scale $(R, C)$, and $do(\ell_2)$ the one obtained by $(\min(\ell_2), \max(\ell_2))$ at $\theta = \frac{1}{2}$. We now show that these are not completely isomorphically equivalent.

Let $(e_n)$ be an orthonormal basis of $\ell_2$, and for each $n \in \mathbb{N}$, let $y^n \in M_n(\ell_2)$ be given by
\[
y^n = \begin{pmatrix}
e_1 & e_2 & \cdots & e_n \\
\vdots & \vdots & \ddots & \vdots \\
e_{n^2 - n + 1} & e_{n^2 - n + 2} & \cdots & e_{n^2}
\end{pmatrix}
\]

We will need the norms different operator space structures of $\ell_2$. Recall that:
\[
\|\sum x_k \otimes e_k\|_{M_n(OH)} = \|\sum x_k \otimes \overline{x_k}\|^{\frac{1}{2}}
\]
\noindent (see \cite{Pisier03}). Simple calculations show that $\|y^n\|_{M_n(R)} = \|y^n\|_{M_n(C)} = \|y^n\|_{M_n(OH)} = n^{\frac{1}{2}}$, and $\|y^n\|_{M_n(\min(\ell_2))} = 1$. To compute the norm of $y^n$ in $\max(\ell_2)$, we recall that $\max(\ell_2) = \min(\ell_2)^*$, and use the following lemma, which may be found in \cite{Paulsen2002} (commentary after Proposition 8.11).

\begin{lemma}\label{lem:5.2}
Let $X$ be an operator space. If $\varphi : X \rightarrow M_n$, then $\|\varphi\|_{cb} \leq n\|\varphi\|$.
\end{lemma}

\begin{lemma}\label{lem:5.3}
We have $\|y^n\|_{M_n(\max(\ell_2))} = n$.
\end{lemma}
\begin{proof}
We have $M_n(\max(\ell_2)) = M_n(\min(\ell_2)^*) = CB(\min(\ell_2), M_n)$ isometrically. Therefore, by the previous lemma:
\[
\|y^n\|_{M_n(\max(\ell_2))} \leq n\|y^n\|_{\ell_2 \rightarrow M_n}
\]

But by \eqref{eq:1.4}:
\begin{eqnarray*}
\|y^n\|_{\ell_2 \rightarrow M_n} & = & \sup\{\|y^n(x)\| : x \in \ell_2, \|x\|=1\} \\
        & = & \sup\{\|(x)_n(y^n)\| : x \in \ell_2, \|x\| = 1\} \\
        & = & \|y^n\|_{M_n(\min(\ell_2))}
\end{eqnarray*}

So $\|y^n\|_{M_n(\max(\ell_2))} \leq n$. But $\|(y^n)_n(y^n)\| = n$, and therefore $\|y^n\|_{M_n(\max(\ell_2))} = n$.
\end{proof}

\begin{theorem}\label{teo:5.4}
The complete extensions $dOH$ and $do(\ell_2)$ are not completely isomorphically equivalent.
\end{theorem}
\begin{proof}
Suppose we have a commutative diagram
  \[
  \xymatrix{ 0 \ar[r] & OH \ar[r]\ar[d]^{A} & dOH
    \ar[r]\ar[d]^{B} & OH\ar[r] \ar[d]^{C} &0 \\ 0 \ar[r] & OH
    \ar[r] & do(\ell_2) \ar[r]& OH\ar[r] &0}
  \]
  
\noindent where the vertical arrows are complete isomorphisms. Then we must have $B(x, y) = (C x, A y + Tx)$, where $x, y \in OH$ and $T : OH \rightarrow OH$ is a completely bounded map. Indeed, since on the first level the twisted sums are trivial with the identity as isomorphism (a consequence of \hyperref[lem:4.7]{Lemma 4.7}), we have for some constant $K > 0$ and for every $x \in OH$:
\[
\|Tx\| \leq K\|(C x, Tx)\|_{do(\ell_2)} = K \|B(x, 0)\|_{do(\ell_2)} \leq K \|B\|\|(x, 0)\|_{dOH} = K \|B\|\|x\|
\]

By the homogeneity of $OH$, $T$ is completely bounded.

By taking the constant function with value $y^n$, we have, for each $n \in \mathbb{N}$:
\[
\|(C_n y^n, T_n y^n)\|_{M_n(do(\ell_2))} = \|B_n(y^n, 0)\|_{M_n(do(\ell_2)))} \leq \|B\|_{cb}\|y^n\|_{M_n(OH)}
\]

For every $n \in \mathbb{N}$, let $F_n(z) = e^{\mu_n (z - \frac{1}{2})}y_n$, where $\mu_n = \log (n^{-1})$. Then, by the previous calculations, $F_n$ is extremal for $y^n$ with respect to the interpolation scheme $(\min(\ell_2), \max(\ell_2)) = OH$, that is, $F_n(\theta) = y^n$ and $\|F_n\|_{M_n(\mathcal{F}(\min(\ell_2), \max(\ell_2)))} = \|y^n\|$. Using \hyperref[lem:4.7]{Lemma 4.7} for $do(\ell_2)$, we have:
\[
\frac{1}{4}(\|C_n y^n\|_{M_n(OH)} + \beta\|T_n y^n - (\log n^{-1}) y^n\|_{M_n(OH)}) \leq \|(C_n y^n, T_n y^n)\|_{M_n(do(\ell_2))}
\]

However
\begin{eqnarray*}
&& \frac{1}{4}(\|C_n y^n\|_{M_n(OH)} + \beta\|T_n y^n - (\log n^{-1}) y^n\|_{M_n(OH)}) \geq \\
&& \frac{1}{4}(\|C^{-1}\|_{cb}^{-1}\|y^n\|_{M_n(OH)} + \beta((\log n) \|y^n\|_{M_n(OH)} - \|T\|_{cb}\|y^n\|_{M_n(OH)})) = \\
&&K(n)\|y^n\|_{M_n(OH)}
\end{eqnarray*}

All this implies that
\[
K(n) \leq \|B\|_{cb}
\]

\noindent for every $n \in \mathbb{N}$, a contradiction.
\end{proof}

We use \hyperref[teo:5.4]{Theorem 5.4} to prove a result that is certainly known, but that we could not find in the literature \hyperref[cor:5.6]{(Corollary 5.6)}. The first part of the following lemma is in \cite{Pisier03}. The second is an easy adaptation of the Banach space case.

\begin{lemma}(Reiteration)\label{lem:5.5}
Let $(X_0, X_1)$ be a compatible couple of operator spaces and let $Y_0 = (X_0, X_1)_{\theta_0}$ and $Y_1 = (X_0, X_1)_{\theta_1}$. Suppose that $X_0 \cap X_1$ is dense in $X_0, X_1$ and $Y_0\cap Y_1$. Then $(Y_0, Y_1)_{\theta} = (X_0, X_1)_{\eta}$ completely isometrically, and the induced complete extensions are completely  projectively equivalent, where $\eta = (1 - \theta)\theta_0 + \theta\theta_1$.
\end{lemma}
\begin{proof}
We prove only the last part. By the 3-lemma \hyperref[prop:2.3]{(Proposition 2.3)}, it is enough to find $T$ completely bounded making the following diagram commute:
  \[
  \xymatrix{ 0 \ar[r] & X_{\eta} \ar[r]\ar[d]^{(\theta_1 - \theta_0)Id} & dX_{\eta}
    \ar[r]\ar[d]^{T} & X_{\eta}\ar[r] \ar@{=}[d] &0 \\ 0 \ar[r] & Y_{\theta}
    \ar[r] & dY_{\theta} \ar[r]& Y_{\theta}\ar[r] &0}
  \]
Let $T$ be given by $T(x, y) = (x, (\theta_1 - \theta_0)y)$ for $x, y \in X_0 \cap X_1$ (recall from \hyperref[teo:3.4]{Theorem 3.4} that this forms a dense subspace of $dX_{\eta}$).

This is well defined. To see this, take $f \in \mathcal{F}(\overline{X})$ with image in $X_0 \cap X_1$ such that $f(\eta) = x$, $f'(\eta) = y$. Also, from the modification of Stafney's Lemma also cited in \hyperref[teo:3.4]{Theorem 3.4} together with Lemma 4.2.3 of \cite{Bergh01}, we may suppose that $f$ is of the form $f(z) = \sum\limits_{i=1}^N c_i(z) x_i$, with $c_i$ continuous bounded on $\mathbb{S}$, analytic on $\mathbb{S}^o$. 

All of this is to ensure that the function $g$ given by $g(z) = f((1- z)\theta_0 + z\theta_1)$ is in $\mathcal{F}(\overline{Y})$. Then $g(\theta) = x$ and $g'(\theta) = (\theta_1 - \theta_0)y$. Therefore, $T$ is well defined.

Also, $g$ as defined above satisfies $\|g\|_{\mathcal{F}(\overline{Y})} \leq \|f\|_{\mathcal{F}(\overline{X})}$, and $T$ is continuous and may be extended to all of $dX_{\eta}$. Since this may be done in $M_n(\mathcal{F}(\overline{Y}))$, we get that $T$ is completely bounded.
\end{proof}

\begin{corol}\label{cor:5.6}
Let $\eta \in [0, 1]$, $\eta \neq \frac{1}{2}$. There exists no $\theta \in [0,1]$ such that $(\min(\ell_2), \max(\ell_2))_{\theta} = R(\eta)$.
\end{corol}
\begin{proof}
Suppose that $(\min(\ell_2), \max(\ell_2))_{\theta} = R(\eta)$. By duality it would follow that $(\min(\ell_2), \max(\ell_2))_{1 - \theta} = R(1 - \eta)$. By the Reiteration Lemma, $dOH$ and $do(\ell_2)$ would be completely equivalent, contradicting \hyperref[teo:5.4]{Theorem 5.4}.
\end{proof}

A remark is in order. By the duality theorem, we have that $(d(R, C)_{\frac{1}{2}})^* = d(C, R)_{\frac{1}{2}}$ completely isomorphically. Also $d(R, C)_{\frac{1}{2}}$ is completely isometric to $d(C, R)_{\frac{1}{2}}$ by $(x, y) \mapsto (x, -y)$ (see the proof of \hyperref[lem:5.5]{Lemma 5.5}). However, it is clear that the canonical identification between $\ell_2$ and its dual is not completely bounded from $d(R, C)_{\frac{1}{2}}$ into $d(C, R)_{\frac{1}{2}}$, by the complete singularity of the induced complete twisted sum and the unicity of $OH$. This may also be proved directly by a reasoning similar to that of \hyperref[prop:4.8]{Proposition 4.8}. Of course, this remark also applies to $(\min(\ell_2), \max(\ell_2))$.

\section{$dOH$ and $do(\ell_2)$ as operator algebras}\label{sec:6}
We recall that an operator algebra $A$ is an operator space that is also a Banach algebra, and such that the inclusion $A \subset B(H)$ giving its operator space structure may be taken respecting the algebraic operations \cite{Blecher01}.

If $A$ is a Banach algebra with an operator space structure, it is an operator algebra if and only if the multiplication $m : A \times A \rightarrow A$ is a completely bounded bilinear operator in the Haagerup sense, that is, the maps $m_n : M_n(A) \times M_n(A) \rightarrow M_n(A)$ given by
\[
m_n((x_{ij}), (y_{ij})) = (\sum\limits_{k} m(x_{ik}, y_{kj}))_{ij}
\]
are uniformly bounded.

If $(X_0, X_1)$ is a compatible couple of operator spaces which are also operator algebras, and such that the multiplication extends to the sum space, then the space $\mathcal{F}(\overline{X})$ is also an operator algebra, and since the quotient of an operator algebra by a closed two sided ideal is also an operator algebra, the interpolation space $X_{\theta}$ is also an operator algebra, $0 < \theta < 1$ \cite{Blecher01}.

\begin{props}\label{prop:6.1}
If $(X_0, X_1)$ is a compatible couple of operator spaces, which are also operator algebras and the multiplication extends to the sum space, then $dX_{\theta}$ is also an operator algebra, for all $\theta \in (0, 1)$.
\end{props}
\begin{proof}
This follows from the previous facts and the identity $dX_{\theta} = \mathcal{F}/(\ker(\delta_{\theta}) \cap \ker(\delta_{\theta}'))$. (Notice that $\ker(\delta_{\theta}) \cap \ker(\delta_{\theta}')$ is a closed two sided ideal of $\mathcal{F}$).
\end{proof}

In particular, since $R$ and $C$ are operator algebras with respect to the natural multiplication of $\ell_2$ \cite{Blecher01}, the twisted sum $dOH$ is also an operator algebra with the multiplication inherited as a quotient of $\mathcal{F}(R, C)$.

However, $\min (\ell_2)$ is not an operator algebra and $OH = (\min(\ell_2), \max(\ell_2))_{\frac{1}{2}}$ is an operator algebra \cite{Blecher01}. We have:
\begin{theorem}\label{teo:6.2}
The complete twisted sum $do(\ell_2)$ is not an operator algebra.
\end{theorem}
\begin{proof}
The multiplication on $do(\ell_2)$ is given by
\[
(x_1, y_1)(x_2, y_2) = (x_1 x_2, x_1 y_2 + y_1 x_2),
\]
for $(x_1, y_1), (x_2, y_2) \in do(\ell_2)$. We must prove that it is not completely bounded.

Again, let $x_n \in M_n(W)$ be the matrix with first line $(e_1 \; ... \; e_n)$ and $0$ elsewhere. We denote the transpose of $x_n$ by $x_n^T$.

We have:
\begin{eqnarray*}
\|x_n\|_{M_n(\min(\ell_2))} = \|x_n^T\|_{M_n(\min(\ell_2))} = 1 \\
\|x_n\|_{M_n(OH)} = \|x_n^T\|_{M_n(OH)} = n^{\frac{1}{4}} \\
\|x_n\|_{M_n(\max(\ell_2))} = \|x_n^T\|_{M_n(\max(\ell_2))} = n^{\frac{1}{2}}
\end{eqnarray*}

Also, $x_n x_n^T$ is the matrix with $e_1 + ... + e_n$ in the first coordinate, and $0$ elsewhere, and therefore has norm $n^{\frac{1}{2}}$.

From the proof of \hyperref[lem:4.5]{Lemma 4.5}, we see that $\|(x_n, \log (n^{-\frac{1}{2}}) x_n)\|_{M_n(do(\ell_2))} = \|(x_n^T, \log (n^{-\frac{1}{2}}) x_n^T)\|_{M_n(do(\ell_2))} = n^{\frac{1}{4}}$.

So, if the multiplication is completely bounded, there is a constant $C > 0$ such that, for every $n \geq 1$, we have
\begin{equation}\label{eq:6.1}
\|(x_n, \log (n^{-\frac{1}{2}}) x_n)(x_n^T, \log (n^{-\frac{1}{2}}) x_n^T))\|_{M_n(do(\ell_2))} \leq C \sqrt{n}
\end{equation}

But, by \hyperref[lem:4.7]{Lemma 4.7}, taking the constant function with value $x_n x_n^T$,
\begin{eqnarray*}
&& \|(x_n, \log (n^{-\frac{1}{2}}) x_n)(x_n^T, \log (n^{-\frac{1}{2}}) x_n^T))\|_{M_n(do(\ell_2))} = \\
&&\|(x_n x_n^T, \log (n^{-1}) x_n x_n^T)\|_{M_n(do(\ell_2))}  \geq\\
&& \frac{1}{4}(\|x_n x_n^T\|_{M_n(OH)} + \beta \|\log (n^{-1}) x_n x_n^T\|_{M_n(OH)}) = \\
&& \frac{1}{4}\sqrt{n}(1 + \beta\log n)
\end{eqnarray*}

Therefore, by equation \eqref{eq:6.1}, the multiplication cannot be completely bounded.
\end{proof}

\section{Some questions}\label{sec:7}
As noted in \hyperref[sec:2]{Section 2}, twisted sums of Banach spaces are defined by $0-$linear maps. If we have a complete twisted sum, we also have a $0-$linear map that defines the twisted sum
\[
  \xymatrix{ 0 \ar[r] & \mathcal{K}\otimes_{\min}Y \ar[r]^{Id_{\mathcal{K}_0}\otimes i} & \mathcal{K}\otimes_{\min}X
    \ar[r]^{Id_{\mathcal{K}_0}\otimes q} & \mathcal{K}\otimes_{\min}Z\ar[r]  &0}
  \]

This $0-$linear map may be obtained by pasting together $0-$linear maps $F_n : M_n(Z) \rightarrow M_n(Y)$.
\begin{question}\label{qu:7.1}
Which sequences of 0-linear maps $(F_n : M_n(Z) \rightarrow M_n(Y))_n$, when pasted together, define a 0-linear map $F : \mathcal{K}\otimes_{\min}Z \rightarrow \mathcal{K}\otimes_{\min}Y$ such that the induced twisted sum is (at least in some sense, completely isomorphic to) an operator space?
\end{question}

\hyperref[prop:2.8]{Proposition 2.8} also raises the following question:
\begin{question}\label{qu:7.2}
Given a Banach space $Y$, for which operator space structures on $Y$ triviality and complete triviality of extension sequences
    \[
    \xymatrix{ 0 \ar[r] & Y \ar[r] & X
    \ar[r] & Z\ar[r]  &0}
    \]
\noindent are equivalent?
\end{question}

Of course, we may ask the same question for $X$ and $Z$.

Recall from \hyperref[sec:4]{Section 4} that for a Banach space $X$, we have the operator space $o(X) = (\min(X), \max(X))_{\frac{1}{2}}$.

\begin{question}\label{qu:7.3}
For which Banach spaces the extension sequence induced by complex interpolation of operator spaces
    \[
    \xymatrix{ 0 \ar[r] & o(X) \ar[r] & do(X)
    \ar[r] & o(X)\ar[r]  &0}
    \]
\noindent is not completely trivial (is completely singular/cosingular)?
\end{question}

\section*{Acknowledgements}
The present work is part of my PhD thesis under supervision of Valentin Ferenczi, whom I would like to thank for his invaluable help. I also would like to thank Gilles Pisier for his helpful remarks with respect to this work.

\bibliographystyle{abbrv}
\bibliography{refs}

\end{document}